\documentclass[reqno,a4paper]{amsart}
\usepackage[xcolor]{changebar}
\usepackage{latexsym}
\usepackage{amssymb,amsthm,amsmath}
\cbcolor{lightgray}
\usepackage{todonotes}

\allowdisplaybreaks 

\theoremstyle{plain}
\newtheorem{theorem}{Theorem}
\newtheorem{lemma}[theorem]{Lemma}

\theoremstyle{definition}

\theoremstyle{remark}
\newtheorem{remark}[theorem]{Remark}

\DeclareMathOperator{\Div }{div}

\def\pa{\partial}
\def\cal{\mathcal}
\let\mib=\boldsymbol

\def\R{{\mathbb R}}
\def\N{{\mathbb N}}

\def\eps{\varepsilon}

\def\malpha{{\mib \alpha}}

\def\mx{{\bf x}}

\def\mxi{{\mib \xi}}

\def\malpha{{\mib \alpha}}

\def\mff{{\mathfrak f}}

\newcommand{\tensoroo}{\mathcal{T}^1_1}
\newcommand{\supp}{\mathrm{supp}\,}

\newcommand{\lara}[2]{\langle #1, #2 \rangle}
\newcommand{\grad}{\mathrm{grad}}
\newcommand{\vphi}{\varphi}

\begin{document}

\title[Galerkin-type methods]{Galerkin-type methods for strictly parabolic equations on compact Riemannian manifolds}
\author{M.\ Graf}\address{Melanie Graf, University of T\"ubingen, 
Department of Mathematics}\email{graf@math.uni-tuebingen.de}
\author{M.\ Kunzinger}\address{Michael Kunzinger, University of Vienna, 
Faculty of Mathematics}\email{michael.kunzinger@univie.ac.at}
\author{D.\ Mitrovic}\address{Darko Mitrovic, University of
Vienna, Faculty of Mathematics}\email{darko.mitrovic@univie.ac.at}



\begin{abstract}
We prove existence of weak solutions to the Cauchy problem corresponding to various strictly parabolic equations 
on a compact Riemannian manifold $(M,g)$. This also includes strictly parabolic equations with stochastic forcing with linear diffusion. Existence is proved through a variant of the Galerkin method and can be used to construct a convergent finite element method.
\end{abstract}
\maketitle

\section{Introduction}

The main subject of this article is the Cauchy problem for a strictly parabolic equation of the form
\begin{align}
\label{main-eq}
\pa_t u +\Div \mff_{\mx}(u)&=\Div(\Div ( A_{\mx}(u) ) ), \ \ \mx \in M,\\
\label{ic}
u|_{t=0} &=u_0(\mx) \in L^2(M), 
\end{align} on a smooth (Hausdorff), orientable, compact $d$-dimensional Riemannian manifold $(M,g)$.  Here, for a fixed $\lambda$, $\mx \mapsto \mff_\mx(\lambda) \in \mathfrak{X}(M)$ is a vector field on $M$ while $\mx\mapsto A_\mx(\lambda)\in \mathcal{T}^1_1(M)$ is a symmetric $(1,1)$ tensor field on $M$.

We suppose that the map $(\mx,\lambda)\mapsto \mff_{\mx}(\lambda) \equiv \mff(\mx,\lambda)$, 
$\mff: M\times \R \to \mathfrak{X}(M)$ is $C^1$ and that, for every $ \lambda\in \R$,
$\mx \mapsto \mff_{\mx}(\lambda)\in \mathfrak{X}(M)$.
Also, $(\mx,\lambda)\mapsto A_\mx(\lambda): M\times\R \to T^1_1M$ is supposed to be $C^2$ and to 
satisfy   $\mx\mapsto A_\mx(\lambda)\in \mathcal{T}^1_1(M)$ for each $\lambda\in \R$ 
and we assume that the  $\lambda$-derivative $\pa_\lambda A_\mx$, which we denote in the sequel by $A'_\mx$ is symmetric positive definite, i.e.\ that the equation \eqref{main-eq} is strictly parabolic. In particular, this also implies that $\lambda \mapsto \langle A_{\mx}(\lambda) \mxi,\mxi \rangle$ 
is strictly increasing for any $\mxi\in T_\mx M$. In other words, we suppose that for a constant {$c>0$} independent of $\lambda$:
\begin{equation}
\label{strict}
\langle  A'_\mx(\lambda) \mxi, \mxi \rangle \geq c \|\mxi\|_g^2, \ \ \mxi \in T_\mx(M). 
\end{equation} 

Before we proceed, we note that the standard form of the diffusion operator in Euclidean space is
\begin{equation}
\label{stand}
\Div a(\mx,u) \nabla u
\end{equation} with a positive definite matrix-valued function $a$ (or very often only non--negative definite; see e.g. \cite{CK} and references therein). 
For $u=u(t,\mx)$ we can rewrite the latter expression as $\Div \Div \big(A(\mx,u(t,\mx))\big)-\Div \big(\Div A(\mx,\lambda)\big|_{\lambda=u(t,\mx)}\big)$ where $A(\mx,\lambda)=\int_0^\lambda a(\mx,z)dz$, which reduces the standard form of the diffusion operator \eqref{stand} to the form from \eqref{main-eq}. We also note that that \eqref{main-eq} locally reduces to a standard parabolic equation (see the proof of Theorem \ref{th:increased_regularity} below) and that the form given in \eqref{main-eq} is chart independent. For further discussions see Remark \ref{rem:Laplace} below.

In the next section, we shall precisely define the notions and notations that we have been using above.

The literature on parabolic equations is vast so we will just mention the classical books \cite{LSU, tay}. They model transport processes governed by convection/advection (first order term) and diffusion (second order term) and therefore are of wide interest in the field of applied mathematics (see e.g. \cite{EK, med, JMN, bio} and references therein). Moreover, they play an important role in geometry since the Ricci flow (via the so called DeTurck's trick) is also an example of a parabolic equation \cite{benett}. 

Interestingly, investigations of general quasilinear parabolic equations on manifolds are fairly rare and recent. In \cite{HP, MM, polden}, one can find results concerning local existence of smooth solutions to parabolic equations on manifolds. As for weak solutions, a two-dimensional situation was considered in \cite{S}, and in a special situation (when the manifold is a torus), one can find an approach similar to ours in \cite[p.\ 319]{tay}).

In the current contribution, we introduce a Galerkin scheme for parabolic PDEs, i.e.\ we consider a finite dimensional approximation to the corresponding Cauchy problem and prove convergence of the resulting sequence of approximate solutions. Once we have developed the method for equation \eqref{main-eq}, we can easily extend it to numerous similar equations. For instance, we can add the Laplace-Beltrami operator on the right-hand side of \eqref{main-eq} and obtain existence of solution for such an equation (see Remark \ref{LB}), and we can extend the result from \cite{S} to manifolds of arbitrary dimension. Finally, we are able to apply our method to parabolic PDEs with stochastic forcing with a linear diffusion term. Parabolic equations (even degenerate parabolic equations) have already been investigated in the stochastic setting in e.g. \cite{Gess} (see also references therein) in the Euclidean space and in spacial situations in \cite{GK} on manifolds (in the form of the vanishing viscosity approximation for stochastic scalar conservation laws). This is the first such result for parabolic equations on manifolds. 
To be more precise, we shall investigate convergence of the Galerkin method for stochastic PDEs of the form
\begin{align}
\label{main-eq-stoch}
\pa_t u +\Div \mff_{\mx}(u)&=\Div(\Div ( A_{\mx}(u) ) )+ \Phi(\mx,u) dW_t, \ \ \mx \in M,
\end{align} on a smooth, compact, orientable, $d$-dimensional (Hausdorff) Riemannian manifold $(M,g)$ where for almost every $\mx \in M$ the matrix valued mapping $\lambda \mapsto A_{\mx}(\lambda)$ is linear. The object $W_t$ is the one-dimensional Wiener process defined on the stochastic basis $(\Omega, {\cal F}, ({\cal F}_t), {\bf P})$ with the sample space $\Omega$, the $\sigma$-algebra ${\cal F}$, the natural filtration ${\cal F}_t$ generated by the Wiener process $W_t$, and the probability measure ${\bf P}$. This means that the unknown function $u$ depends on $(t,\mx,\omega)\in \R^+\times M\times \Omega$.

The paper is organized as follows. In Section \ref{diffprelsec} we introduce basic notions from Riemannian geometry and Sobolev spaces on manifolds. We note that although the theory of Sobolev spaces on manifolds is well investigated, we find that this part of the paper represents a very useful collection of (not easy to find) information, especially from the viewpoint of the rapidly developing field of PDEs on manifolds.
In Section \ref{sec:galerkin}, we introduce the Galerkin scheme that we are going to use and prove the convergence of this 
scheme for bounded flux $\mff_\mx$ and diffusion $A_\mx$ under appropriate assumptions on these bounds (see \eqref{bnd-1}). 
Section \ref{sec:unbounded_flux} is devoted to proving existence of solutions under so called geometry-compatibility conditions \eqref{geomcomp} (which are actually incompressibility conditions when it comes to fluid dynamics as explained in \cite[Introduction]{GKM}), but without assuming \eqref{bnd-1}. In the next two remarks, we explain how to deal with the non-strictly parabolic equations on manifolds regularized with the Laplace-Beltrami operator \cite[Section 4]{GKM}, and how to obtain solutions and their regularity properties for parabolic equations of higher order (see also \cite{S}). In the final Section \ref{sec:stochastic}, we prove existence of solutions to \eqref{main-eq-stoch} with initial data \eqref{ic}. In the process, we recall some basic notions from stochastic calculus.
\section{Preliminaries from Riemannian geometry and functional analysis}
\label{diffprelsec}

We denote by $\mathcal{T}^1_0(M)$ the set of vector fields over $M$, by $\Omega^1(M)\equiv\mathcal{T}_1^0(M)$ the set of one forms 
over $M$, and by $\mathcal{T}^1_1(M)$ the set of $(1,1)$-tensors over $M$.

If $X\in \mathcal{T}^1_0=\mathfrak{X}(M)$ is a $C^1$ vector field on
$M$ with local representation $X=X^i\frac{\pa}{\pa x^i}$, then its divergence $\Div X\in C(M)$
is locally given by
\begin{equation}\label{divx}
\Div X = \frac{\pa X^k}{\pa x^k} + \Gamma^j_{kj}X^k.
\end{equation}
In this expression, the $\Gamma^i_{kl}$ denote the Christoffel symbols corresponding to $g$,
\[
\Gamma^i_{kl} = \frac{1}{2} g^{im}\Big(\frac{\partial g_{mk}}{\partial x^l} + \frac{\partial g_{ml}}{\partial x^k} - \frac{\partial g_{kl}}{\partial x^m} \Big),
\]
$g^{im}$ are the coefficients of the inverse metric, and throughout this paper
we always adhere to the Einstein summation convention. 
The same expression holds for $X$ a distributional vector field, and similarly for the formulae given below, which  
we formulate in the smooth case with the understanding that they carry over by continuous extension also to 
the distributional setting (cf.\ \cite{Mar,GKOS}). 

Using abstract index notation (cf., e.g., \cite[Ch.\ 2.4]{wald}), for a vector field $X^a$ we have
\[
\Div X = \nabla_a X^a.
\]
If a $C^1$ one-form $\omega\in \mathcal{T}^0_1(M)=\Omega^1(M)$ is locally given by $\omega=\omega_idx^i$, then
its divergence is defined as the {\em metric} contraction of its covariant differential $\nabla \omega\in \mathcal{T}^0_2(M)$, so
$\Div \omega = \nabla^b\omega_b$ in abstract index notation, or
\begin{equation} \label{divom}
\Div \omega = g^{ij} \pa_i\omega_j - \Gamma^k_{il} g^{il}\omega_k.
\end{equation} 
in local coordinates.
If $T = {T^a}_b\in \tensoroo(M)$, then $\Div T = \nabla_a  {T^a}_b\in \Omega^1(M)$. Locally, 
$T=T^k_i \frac{\pa}{\pa x^k}\otimes dx^i$. Then 
$\Div T = (\Div T)_i dx^i$, where
\begin{equation}\label{div11}
(\Div T)_i = \pa_j T^j_i + \Gamma^j_{jl} T^l_i - \Gamma^l_{ji} T^j_l.
\end{equation}
For $T\in \tensoroo(M)$ in $C^2$, $\Div(\Div(T))\in C(M)$. The explicit form of $\Div(\Div(T))$ in terms of local coordinates 
can be found in \cite{GKM} (it will not be needed here).
\begin{remark}\label{rem:Laplace}
Using the notations introduced above, we can add to the observations concerning the relationship between \eqref{main-eq}
and more standard forms of parabolic operators initiated after \eqref{stand}, now in the full setting of Riemannian manifolds. 
First, denote by $\delta\in \mathcal{T}^1_1(M)$ the Kronecker-delta tensor,
$\delta(\omega,X) = \omega(X)$ for $\omega\in \Omega^1(M)$ and $X\in \mathfrak{X}(M)$ (cf.\ \cite{AM}), and let $A_\mx(\lambda):= \lambda \cdot \delta(\mx)$. In any local chart, $\delta = \pa_i\otimes dx^i$, and using
\eqref{div11} we see that $\Div(A_\mx(u)) = \omega_i dx^i$, where 
\begin{align*}
\omega_i = \delta^j_i\partial_j u + \big(\Gamma^j_{jl}\delta^l_i - \Gamma^l_{ji} \delta^j_l \big) u = \delta^j_i \partial_j u = \partial_i u
\end{align*} 
Since $(\nabla u)^i = g^{ij}\partial_j u$, we get $\omega_i = g_{ik}(\nabla u)^k$. Hence
in this special case we have $\Div(A_\mx(u)) = g \nabla u = du$, and thereby
\begin{equation}\label{eq:Laplace}
\Div(\Div(A_\mx(u))) = \Div(du) = \nabla^a \nabla_a u = \Delta u,
\end{equation}
with $\Delta$ the Laplace-Beltrami operator on $M$. 
Clearly, $\Div(A_\mx(u))$ in general will not necessarily be of such a simple form. However, our analysis of
\eqref{main-eq} readily carries over to right hand sides of the form $\Div(C(x,u) \nabla u)$ ($C(.,\lambda)\in \mathcal{T}^1_1(M)$) as well,
cf.\ also Remark \ref{rem:general_Q} below.
\end{remark}
The metric $g$ induces scalar products on any tensor product $(T_{\mx}M)^r_s$, given in abstract index notation by
\[
\langle S,T\rangle = S^{a_1\dots a_r}{}_{b_1\dots b_s} T_{a_1\dots a_r}{}^{b_1\dots b_s}.
\] 
We will denote the corresponding norms by $\|\,\|_g$ irrespective of $r$ and $s$. Also, for any tensor field
 $T\in \mathcal{T}^r_s(M)$ we set
\[
\|T\|_{L^\infty(M)} := \sup_{\mx \in M} \|T(\mx)\|_g.
\]
and
\[
\|T\|_{L^2(M)} := \Big(\int_M \|T(\mx)\|^2_g\, dV(\mx) \Big)^{\frac{1}{2}}.
\]
Here, $dV$ is the oriented Riemannian volume measure (given, in any chart of the oriented atlas, by
$dV = \sqrt{|\det(g_{ij})|}\,dx^1\wedge \dots \wedge dx^n$).

In the Cauchy problem \eqref{main-eq}, \eqref{ic}, $(\mx,\lambda)\mapsto A_\mx(\lambda): M\times\R \to T^1_1M$ is $C^2$ and 
for each $\lambda\in \R$, $\mx\mapsto A_\mx(\lambda)\in \mathcal{T}^1_1(M)$. 
We shall need the following lemma.
\begin{lemma}
\label{divdiv-IbP}
Let $A\in {\cal T}^1_1(M)$ be $C^2$, $f\in C^2(M)$, $M$ compact and oriented without boundary. Then
\begin{equation}
\label{div-div-rule}
\int_M f \Div(\Div A) dV 
=-\int_M (\Div A)(\nabla f) dV.
\end{equation} 
\end{lemma}
\begin{proof}
By Stokes' theorem, $\int_M \Div X \,dV = 0$ for any $C^1$ vector field $X$. 
If $\omega=X_a$ is the one-form metrically equivalent to $X$, $\omega=X^\flat$, then
$\Div \omega = \nabla^a X_a = \nabla_a X^a$, so we also have  $\int_M \Div \omega \,dV = 0$
for any $C^1$ one-form $\omega$. Now
\[
\Div(f \Div A) = \nabla^b(f\nabla_a {A^a}_b) = \nabla^b f \nabla_a {A^a}_b + f \nabla^b\nabla_a {A^a}_b
= (\Div A)(\nabla f) + f\Div(\Div A),
\]
so
\[
0 = \int_M \Div(f\Div A)\,dV = \int_M (\Div A)(\nabla f)\,dV + \int_M f \Div(\Div A)\,dV.
\]
\end{proof}
Furthermore, under the same assumptions as in Lemma \ref{divdiv-IbP} we have:
\begin{lemma}\label{lem:trace_identity}
\begin{equation}\label{eq:divdiv-trace}
\int_M f \Div(\Div A)) dV = \int_M \mathrm{tr}(\tilde H^f\circ A^\top)\,dV = \int_M \mathrm{tr}(A\circ \tilde H^f)\,dV.
\end{equation}
Here, $\mathrm{tr}$ denotes the trace, ${(A^\top)^a}_b = {A_b}^a$ is the transpose of $A$,
and ${(\tilde{H}^f)^a}_b=\nabla^a \nabla_b f$ is the $(1,1)$-tensor field metrically equivalent
to the Hessian $\nabla_a\nabla_b f$ of $f$.
\end{lemma}
\begin{proof} We have
\begin{align*}
(\Div A)(\nabla f) &= \nabla_a {A^a}_b \nabla^b f = \nabla_a ({A^a}_b \nabla^b f) - {A^a}_b(\nabla_a \nabla^b f)\\
&= \Div(A(\nabla f)) - (A^\top)_b{ }^a(\nabla_a \nabla^b f)\\
&= \Div(A(\nabla f)) - (A^\top)^b{ }_a(\nabla^a \nabla_b f)\\
&= \Div(A(\nabla f)) - \mathrm{tr}(\tilde H^f\circ A^\top).
\end{align*}
Integrating this identity over $M$, and applying Stokes' theorem and Lemma \ref{divdiv-IbP} gives the first equality.
The second one follows since $\tilde H^f$ is symmetric and $\mathrm{tr}B = \mathrm{tr}B^\top$ for any $(1,1)$-tensor $B$.
\end{proof}
\begin{remark}\ 
\begin{itemize}
\item[(i)] Since $\Div \Div A = \Div \Div A^\top$ and $A^{\top\top} = A$, on the right hand side of \eqref{div-div-rule} we
may replace $A$ by $A^\top$, and conversely for \eqref{eq:divdiv-trace}.
\item[(ii)] By density, the conclusion of Lemma \ref{divdiv-IbP}  remains valid
for $A\in H^2(M)$ and $f\in H^1(M)$, and that of Lemma \ref{lem:trace_identity} for $A\in H^2(M)$ and $f\in H^2(M)$.
\end{itemize}
\end{remark}

We need the following consequence of the parabolicity of the operator $\Div\Div A_{\mx}(\cdot)$. However, we first need to assume that the vector field $\mff$ and the tensor $A_{\mx}$ are bounded in the sense that there exists $\bar{C}>0$ such that (recalling
that we set $A'_{\mx}(\lambda):=\pa_{\lambda} A_{\mx}(\lambda) $) for all $\lambda$:
\begin{equation}
\label{bnd-1}
\begin{split}
\|\mff_{\mx}(\lambda) \|_{L^\infty(M)} +\| A_{\mx}(\lambda) \|_{L^\infty(M)} &+ \|\lambda  A'_{\mx}(\lambda) \|_{L^\infty(M)}  \\
&  +  \| \Div A_{\mx}(\lambda) \|_{L^\infty(M)} \leq \bar{C} (1+|\lambda|) 
\end{split}
\end{equation}
As we shall see in Section 4, the assumption \eqref{bnd-1} can be avoided in some important situations. 
\begin{lemma}
\label{Lpar}
(i) For any $u\in H^2(M)$, 
\begin{equation}
\begin{split}
\label{P1}
\int_M \Div & \Div A_{\mx}(u)\,  u(\mx) dV(\mx)\\ 
&\leq -c \int_M \|\nabla u(\mx)\|_g^2 \, dV(\mx) + C \max\{ 1, \int_M |u(\mx)|^2\, dV(\mx) \}.
\end{split}
\end{equation} 

(i') If $\lambda\mapsto A_\mx(\lambda)$ is linear then
\begin{equation}
\begin{split}
\label{P1-1}
\int_M \Div & \Div A_{\mx}(u)\,  u(\mx) dV(\mx)\\ 
&\leq -c \int_M \|\nabla u(\mx)\|_g^2 \, dV(\mx) + C \int_M |u(\mx)|^2\, dV(\mx).
\end{split}
\end{equation} 

(ii) If $\lambda\mapsto A_\mx(\lambda)$ is linear and $u\in H^3(M)$, then (with $\Delta u = \Div \grad u$)
\begin{equation}\label{P2-L}
\begin{split}
\int_M \Div \Div & A_{\mx}(u)\,  \Delta u(\mx) dV(\mx)
\geq k \int_M |\Delta u(\mx)|^2 \, dV(\mx)\\ 
&- K \max\{ 1, \int_M |u(\mx)|^2\, dV(\mx), \int_M |\nabla u(\mx)|^2\, dV(\mx) \}
\end{split}
\end{equation}
for some constants $k,K,c,C>0$ independent of $u$.
\end{lemma}
\begin{proof} {By density, it suffices to assume $u$ to be smooth.}

(i) By \eqref{div-div-rule}, we have
\begin{equation}
\label{P3}
\int_M \Div \Div A_{\mx}(u)\, u(\mx) dV(\mx) =-\int_M  \Div A_{\mx}(u)\,  \nabla u(\mx) dV(\mx).
\end{equation}
Here, 
\begin{align*}
\Div (A_{\mx}(u))\nabla u &= \nabla_b (A_{\mx}(u))^b{}_a \nabla^a u =  (\nabla_b A^b{}_a)(\mx,u(\mx))\nabla^a u 
+ (A')^b{}_a \nabla_b u \nabla^a u\\
&= \big(\Div A_{\mx}(\lambda)\big|_{\lambda=u}\big)(\nabla u) + \lara{A'_{\mx}(u)\nabla u }{\nabla u}.
\end{align*}
By \eqref{bnd-1} we have $|\big(\Div A_{\mx}(\lambda)\big|_{\lambda=u}\big)(\nabla u)| \le \bar C(1+|u|)\|\nabla u\|_g$.
Combining this with \eqref{strict}, we obtain
\begin{align}
\label{new-11}
-\Div (A_{\mx}(u))\nabla u &\le - c \|\nabla u\|_g^2 + \bar C (1+|u|)\|\nabla u\|_g \\
& \le - c \|\nabla u\|_g^2 + \bar C^2 N + \frac{1}{N} \|\nabla u\|_g^2 + \bar CN|u|^2 + \frac{\bar C}{N} \|\nabla u\|_g^2,
\nonumber
\end{align}
where we applied the Peter-Paul inequality twice in the last step. By choosing  $N>0$ large enough, we conclude the proof. 

(i') We get the conclusion by noticing that linearity of $A_\mx$ implies {due to} \eqref{bnd-1}:
$$
\| \Div (A_{\mx}(\lambda)) \|_{L^\infty(M)} \leq C \lambda
$$
 with which we omit the term $\bar C^2 N$ in \eqref{new-11}.

(ii) In this case we have $A_\mx(u) = A_\mx\cdot u$, and by Stokes' theorem
\[
\int_M \Div \Div ( A_{\mx}u)\,  \Delta u(\mx)\, dV(\mx) = - \int_M  \Div ( A_{\mx}u)( \nabla \Delta u)\, dV(\mx). 
\]
We have
\begin{align*}
\Div ( A_{\mx}u)( \nabla \Delta u) = u \nabla_a A^a{}_b \nabla^b\nabla^c\nabla_c u + A^a{}_b\nabla_a u \nabla^b\nabla^c\nabla_c u,
\end{align*}
where (with $R$ the Riemann tensor, cf., e.g., \cite[(3.2.3)]{wald})
\begin{align*}
A^a{}_b\nabla_a u \nabla^b\nabla^c\nabla_c u = A^a{}_b\nabla_a u \nabla^c\nabla^b\nabla_c u
+ A^a{}_b\nabla_a u R^{cb}{}_c{}^d \nabla_d u. 
\end{align*}
Again by Stokes' theorem,
\begin{align*}
\int A^a{}_b\nabla_a & u \nabla^c\nabla^b\nabla_c u\, dV = - \int \nabla^c( A^a{}_b \nabla_a u) \nabla^b\nabla_c u\,dV \\
&= - \int (\nabla^c  A^a{}_b)\nabla_a u \nabla^b\nabla_c u \, dV - \int  A^a{}_b\nabla^c\nabla_a u \nabla^b\nabla_c u \,dV.
\end{align*}
The integrand in the last term can be rewritten as
\[
\nabla^c\nabla_a u A^a{}_b \nabla^b\nabla_c u = \mathrm{tr}(\tilde H^u\cdot A \cdot \tilde H^u) = \mathrm{tr}((\tilde H^u)^2\cdot A)
\]
Collecting terms, we arrive at 
\begin{align} 
\label{nice}
\int_M \Div  \Div ( A_{\mx}u)\, & \Delta u(\mx)\, dV(\mx) = \int_M \mathrm{tr}((\tilde H^u)^2\cdot A) \,dV - \int_M u \Div(A)(\nabla \Delta u)\,dV\\
& +\int_M (\tilde H^u)^b{}_c \nabla^c A^a{}_b \nabla_a u\, dV - \int_M A^a{}_b\nabla_a u R^{cb}{}_c{}^d \nabla_d u\, dV.
\nonumber
\end{align} 
Let us consider each term in the previous expression individually. In general, if $R$ is a symmetric 
matrix and $S$ is symmetric and positive definite, then (with $\|\,.\,\|_F$ the Frobenius norm) we have 
\[
\mathrm{tr}(R^2\cdot S) \ge \lambda  \mathrm{tr}(R^2)=\lambda \|R\|^2_F, 
\]
where $\lambda$ is the smallest eigenvalue of $S$. 
 Thus there exists a constant $k>0$ such that
\begin{equation}
\label{nice1}
\int \mathrm{tr}((\tilde H^u)^2\cdot A) \,dV \geq k \int_M  \|\tilde H^u\|^2_F \, dV.
\end{equation} 
For the second term, we have (with $\dim(M)=d$)
\begin{equation}
\label{nice2}
\begin{split}
\big|\int_M u &\Div(A)(\nabla \Delta u)\,dV \big| \leq \|u \Div(A)\|_{H^1} \|\nabla \Delta u \|_{H^{-1}(M)}\\
&\leq {K_0} \|u \Div(A)\|_{H^1} \| \Delta u \|_{L^2(M)}   
\\&\leq \tilde{K}_1 \|u \Div(A)\|^2_{H^1} +\frac{k}{4d^2} \int_M |\Delta u|^2 dV \leq K_1 \|u \|^2_{H^1} +\frac{k}{4d^2} \int_M |\Delta u|^2 dV, 
\end{split}
\end{equation}
for an appropriate constant $K_1$ (and $k$ from \eqref{nice1}). We estimate the third term in a similar manner:
\begin{equation}
\label{nice3}
\begin{split}
\big|\int_M (\tilde H^u)^b{}_c & \nabla^c A^a{}_b \nabla_a u\, dV \big|
\\&
{
\leq \frac{k}{4a} \int_M \|\tilde H^u\|_g^2 dV+ K_2 \|\nabla A\|^2_{\infty} \int_M  \|\nabla u\|_g^2 dV}\\
&
{
\leq \frac{k}{4} \int_M  \|\tilde H^u\|^2_F  \, dV+ K_2 \|\nabla A\|^2_{\infty} \int_M  \|\nabla u\|_g^2\, dV.}
\end{split} 
\end{equation}
{Here 
$a$ is an appropriate constant to estimate $\|\,.\|_g$ by $\|\,.\,\|_F$ on $M$.}
The last term is estimated directly:
\begin{equation}
\label{nice4}
\big| \int_M A^a{}_b\nabla_a u R^{cb}{}_c{}^d \nabla_d u\, dV \big| \leq 
{K_3} \int_M \|\nabla u\|_g^2 dV. 
\end{equation} 
Combining \eqref{nice}--\eqref{nice4}  and using $\frac{3k}{4} \int_M   \|\tilde H^u\|^2_F  \, dV\geq \frac{3k}{4d^2} \int_M  |\Delta u|^2 dV $,  we obtain \eqref{P2-L}.
\end{proof}

Let us now recall some basic properties of Sobolev spaces on manifolds that we shall need in the sequel,
referring to \cite{[CP],[H]} for basic definitions.
As before, let $(M,g)$ be a compact orientable Riemann manifold, and let, for $u\in {\cal D}'(M)$
$$
\Delta u:= \Div \grad(u),
$$ 
so $\Delta$ denotes the Laplace-Beltrami operator.
Denoting by $\delta$ the codifferential ($\delta \alpha := (-1)^{nk+n+1}*\!d\!*\!\alpha$ for $\alpha\in \Omega^k(M)$, $*$ the
Hodge-star operator), this corresponds to $-\Delta=(d+\delta)^2=d \delta+ \delta d$ on forms ($-\Delta =\delta d$ on functions). 
Then $-\Delta$ is a positive, essentially self-adjoint operator on $L^2(M)$, cf.\ \cite{[G]}. 
Moreover, it is elliptic and formally self-adjoint. The same is therefore true of $I-\Delta$. Let $\bar{\Delta}$ be the closure of $\Delta$, 
then $-\bar{\Delta}$ has all the above properties and is in fact self-adjoint. We remark that all of the above is true if $M$ is merely complete. 
By \cite{[S]}, cf. \cite[p.298.]{[T]}, the operator $(I-\bar{\Delta})^{s/2}$ is a pseudo-differential operator of order $s$ for any $s\in \R$. Moreover, it is itself elliptic, self adjoint and positive with inverse $(I-\bar{\Delta})^{-s/2}$. Often (see e.g. \cite{[T]}), no distinction is made between $\Delta$ and $\bar{\Delta}$.

Set $\Lambda^s:=(I-\bar{\Delta})^{s/2}$, then $(\Lambda^s)^{-1}=\Lambda^{-s}$ and by \cite[Th.\ 8.5]{[CP]}, for every $r\in \R$:
$$
\Lambda^s: H^r(M)\to H^{r-s}(M)
$$ is a linear isomorphism with inverse $\Lambda^{-s}$.

Now, $\Lambda^{-1}: L^2(M)\to H^1(M)\hookrightarrow L^2(M)$ is a compact operator (by Rellich's theorem), so by the spectral theorem for compact operators there exists a countable orthonormal basis $(e_n)$ 
of $L^2(M)$ consisting of eigenfunctions of $\Lambda^{-1}$, and we denote the 
eigenvalue of $e_n$ by $\lambda_n^{-1}$. This implies $\Lambda^1 e_n=\lambda_n e_n$ and, since $\Lambda^1$ is elliptic, we have $e_n \in C^\infty(M)$. {Also, $\lambda_n \nearrow +\infty$ as
$n\to\infty$.}

Since $\Lambda^{-s}:L^2(M)\to H^s(M)$ is an isomorphism, we may introduce on $H^s(M)$ an equivalent scalar product by 
\begin{equation}
\label{SP}
\langle u, v \rangle_s:=\langle \Lambda^s u, \Lambda^s v \rangle_{L^2}.
\end{equation}

With respect to $\langle\cdot,\cdot \rangle_s$, $(e_m)_{m\in \N}$ is an orthogonal system:
\begin{align}\label{eq:e_j_sobolev_products}
\langle e_m,e_n \rangle_s= \langle \Lambda^s e_m,\Lambda^s e_n \rangle_{L^2}=\langle \Lambda^{2s} e_m, e_n  \rangle_{L^2}
=\langle \lambda_m^{2s} e_m,e_n\rangle_{L^2}=\lambda_m^{2s} \delta_{mn},
\end{align} 
and so the corresponding orthonormal system is
\begin{equation}\label{eq:Hs_basis}
(e^{(s)}_m)_{m\in \N} := (\lambda_m^{-s}e_m)_{m\in\N}.
\end{equation}
It is complete in $H^s(M)$ since 
$$
\langle u, e_m \rangle_s=0 \ \ \forall m\in \N \, \implies \, \lambda^{-2s}_m \langle 
\Lambda^{2s}u,e_m \rangle_{L^2}=0 \ \ \forall m \; \implies \Lambda^{2s}u=0 \, \implies \, u=0.
$$ 

Moreover, we have
$$
u\in H^s(M) \, \Longleftrightarrow \Lambda^s u\in L^2 \Longleftrightarrow 
\sum\limits_{m\in \N} |\langle \Lambda^su, e_m \rangle_{L^2}|^2=\sum\limits_{m\in \N} |\langle u,e_m \rangle_{L^2}|^2 \lambda_m^{2s}<\infty.
$$
An important consequence of the above, which we shall make use of repeatedly below, is that projections onto $\mathrm{span}(e_1,\dots,e_n)$ do 
not depend on the Sobolev-index:
\begin{equation}\label{eq:projection_independent_of_s}
\sum_{m=1}^n \langle u,e_m\rangle_{L^2} e_m = \sum_{m=1}^n \langle u,e_m^{(s)}\rangle_{s} e^{(s)}_m.
\end{equation}
Let us show that the standard scalar product in $H^1(M)$ and the one defined in \eqref{SP} coincide. By \cite[Def.\ 2.1]{[H]},  
$$
\langle u, v \rangle_{H^1}=\langle u,v \rangle_{L^2}+\langle \nabla u, \nabla v \rangle_{L^2}
$$ 
and by Green's second identity 
\[
\langle \nabla u,\nabla v \rangle_{L^2} = \int_M \langle \nabla u, \nabla v \rangle\, dV=
-\frac{1}{2}\int_M  v\Delta u +  u \Delta v \, dV
\]
Consequently, 
\begin{align*}
2\langle u,v \rangle_{H^1} &= 2\langle u, v \rangle_{L^2} - \langle \Delta u, v \rangle_{L^2} - \langle u,  \Delta v \rangle_{L^2}
=\langle (I - \Delta)u,v \rangle_{L^2}+\langle u, (I - \Delta) v \rangle_{L^2}\\
&  = 2\langle \Lambda^2u,v \rangle_{L^2}
= 2  \langle\Lambda u, \Lambda v \rangle_{L^2}=2 \langle u, v \rangle_1.
\end{align*}
As we shall make use of certain variants of vector-valued Sobolev spaces, we conclude this section by
recalling some basic definitions and properties, referring to \cite{kreut} for details and references.
For any $s\in \R$, $T>0$, we define $L^2((0,T),H^s(M))$ to be the space of measurable functions
$u:(0,T)\to H^s(M)$ such that 
\[
\|u\|_{L^2((0,T),H^s(M))} := \Big(\int_0^T \|u(t,\,.\,)\|_{H^s(M)}^2\,dt \Big)^{1/2} < \infty.
\]
By $H^1((0,T),H^s(M))$ we denote the space of all $u\in L^2((0,T),H^s(M))$ that are weakly
differentiable (with respect to $t$) and such that
\[
\|u\|_{H^1((0,T),H^s(M))}^2 :=  \|u\|_{L^2((0,T),H^s(M))}^2 + \|\pa_t u\|_{L^2((0,T),H^s(M))}^2 <\infty.
\] 

Finally, we shall require the following fundamental result (cf.\ \cite[Th.\ II.5.16]{BoyFab}):
\begin{theorem}\label{th:aubin} (Aubin-Lions-Simon) 
	Let $E_1\subset\subset E\subset E_0$ be Banach spaces such that $E_1$ is compactly embedded in $E$ 
	and $E$ is continuously embedded in $E_0$. For $1\leq p,q \leq \infty$, let
	$$
	W=\{ u\in L^p((0,T); E_1):\, \pa_t u \in L^q((0,T);E_0) \}.
	$$Then,
	\begin{itemize}
		\item[(i)] If $p  < \infty$, then the embedding of $W$ into $L^p((0,T); E)$ is compact.
		
		\item[(ii)] If $p  = \infty$ and $q  >  1$, then the embedding of $W$ into $C([0, T]; E)$ is compact.
	\end{itemize}
\end{theorem}

\section{Galerkin approximation}\label{sec:galerkin}

We fix an orthonormal basis $\{ e_k\}_{k\in \N}$ in $L^2(M)$ consisting of eigenfunctions of the Laplace-Beltrami operator,  
as described in Section \ref{diffprelsec}
and look for approximate solutions to \eqref{main-eq}, \eqref{ic} in the form
\begin{equation}
\label{app-1}
u_n(t,\mx)=\sum\limits_{k=1}^n \alpha^n_k(t) e_k(\mx).
\end{equation} 

Next, we insert $u_n$ from \eqref{app-1} into \eqref{main-eq} and look for $\alpha_k$, $k=1,\dots, n$, so that \eqref{main-eq} is 
satisfied in the space $\mathrm{span}( e_1,\dots, e_n)$. In other words, we multiply the expression
\begin{equation}
\label{app-12}
\pa_t u_n +\Div \mff_{\mx}(u_n)=\Div(\Div ( A_{\mx}(u_n) ) )
\end{equation} 
by $e_j$ for every $j=1,\dots, n$, and integrate over $M$. 
We get after taking into account orthonormality of the basis $(e_k)_{k\in \N}$:
\begin{equation}
\label{app-2}
\begin{split}
\dot{\alpha}^n_j(t)& =\int_M \Big\langle \mff_{\mx}(\sum\limits_{k=1}^n {\alpha_k^n}(t) e_k(\mx)),  \nabla e_j(\mx)\Big\rangle  dV(\mx) \\
&+\int_M \mathrm{tr} \Big(A_{\mx}\Big(\sum\limits_{k=1}^n {\alpha_k^n}(t) e_k(\mx)\Big)  \tilde{H}^{ e_j}(t,\mx)\Big) dV(\mx).
\end{split}
\end{equation}

Using Lemma \ref{divdiv-IbP} and conditions \eqref{bnd-1}, one can prove existence of solutions to the latter system supplemented with appropriate initial data. Indeed, the standard Cauchy theorem provides existence of a local solution and conditions \eqref{bnd-1} enable us to extend the solution for an arbitrary time interval as shown in e.g. \cite[Theorem 5.2.1]{Oks} even in the stochastic case, which we shall consider later. More precisely, we have the following lemma.

\begin{lemma}
\label{L1} Under the assumption \eqref{bnd-1}, for any fixed $n\in \N$, the system of ODEs \eqref{app-2} with initial data $\alpha_j(0)=\alpha_{j0} \in \R$ has a globally defined solution.
\end{lemma}
\begin{proof}
According to the Cauchy theorem, system \eqref{app-2} with the given initial data has a solution defined on $[0,T)$ for some $T>0$. 
Then applying the bounds \eqref{bnd-1} to \eqref{app-2}, we obtain an estimate of the form
\[
|\alpha(t)| \le |\alpha(0)| + \bar C \int_0^t (1+|\alpha(s)|)\,ds,
\]
where $\alpha = (\alpha_1,\dots,\alpha_n)$ and $|\,|$ is any norm on $\R^n$. Consequently, 
$|\alpha(t)| \le (|\alpha(0)|+\bar C t) e^{\bar C t}$ by Gronwall's inequality. Hence
the functions $\alpha^n_j$, $j=1,\dots, n$, cannot blow up as $t\nearrow T$, 
implying that they can be extended to all of $\R$.
\end{proof}

If we rewrite \eqref{ic} as
$$
u_0(\mx)=\sum\limits_{k=0}^\infty \alpha_{k0} e_k(\mx)
$$ and take $\alpha^n_j(0)=\alpha_{j0}$ in \eqref{app-2}, we know that there exists a sequence of approximate 
solutions to \eqref{main-eq}, \eqref{ic} in the sense that \eqref{app-2} is satisfied for any $n\in \N$. 
Denote this sequence by $(u_n)$. We want to prove that it converges strongly in $L^2((0,T)\times M)$ 
(for any $T>0$).

\begin{lemma}
\label{L2}
The sequence $(u_n)$ is bounded in $L^2((0,T);H^1(M))$.
\end{lemma}
\begin{proof}
We first multiply \eqref{app-12} by $u_n$ and integrate over $(0,t)\times M$ 
for any fixed $t\in (0,T)$. We have after integration by parts (cf.\ Lemma \ref{divdiv-IbP}):
\begin{equation}
\label{H1-1}
\begin{split}
&\frac{1}{2}\int_M  |u_n(t,\mx)|^2  dV(\mx) - \frac{1}{2}\int_M |u_0(\mx)|^2 dV(\mx) 
\\ &+\int_0^t \int_M \Div(A_{\mx}(u_n(\tau,\mx)))\cdot\nabla u_n(\tau,x)\, dV(\mx)d\tau \\
&  \hspace*{7em} =\int_0^t \int_M \mff_{\mx}(u_n(\tau,\mx)) 
\cdot\nabla u_n(\tau,\mx) d\mx d\tau.
\end{split}
\end{equation} Abbreviating $u_n(\tau,\mx)$ by $u_n$ henceforth, Lemma \ref{Lpar} (i) shows
\begin{align*}
\int_0^t \int_M \Div(A_{\mx}(u_n))\cdot\nabla u_n \, dV(\mx)d\tau \geq c  \int_0^t & \int_M   \| \nabla u_n \|_g^2 dV(\mx)d\tau\\
&- CT \mathrm{vol}(M) -C\int_0^t \int_M | u_n |^2\, dV(\mx) d\tau.
\end{align*}
Also, the Peter-Paul inequality gives
\begin{align*}
\int_0^t \int_M \mff_{\mx}(u_n) \cdot\nabla u_n \, dV(\mx) d\tau \leq 
\frac{c}{2}\int_0^t \int_M \| & \nabla u_n \|_g^2\, dV(\mx)d\tau\\ 
&+ \frac{2}{c}\int_0^t \int_M \|\mff_{\mx}(u_n)\|_g^2\, dV(\mx) d\tau.
\end{align*} 
Inserting this into \eqref{H1-1} and using \eqref{bnd-1} we see that for certain constants $C_1, C_2>0$
\begin{equation}\label{eq:18}
\begin{split}
\frac{1}{2}\int_M  & |u_n(t,\mx)|^2\, dV(\mx) + \frac{c}{2}\int_0^t \int_M \|\nabla u_n\|^2 \, dV(\mx)\\
& \le CT \mathrm{vol}(M) + \frac{1}{2}\int_M |u_0(\mx)|^2 dV(\mx) + \frac{2}{c} \int_0^t \int_M \|\mff_{\mx}(u_n)\|_g^2 dV(\mx)\\ 
& \hspace*{5em} + \int_0^t \int_M |u_n|^2 dV(\mx) d\tau 
\le C_1 + C_2 \int_0^t \int_M |u_n|^2 dV(\mx) d\tau
\end{split}
\end{equation}
Therefore Gronwall's inequality implies that $\int_M  |u_n(t,\mx)|^2\, dV(\mx) \le C_T$ for some constant $C_T$
depending only on $T$. Re-inserting into \eqref{eq:18}, we arrive at the desired conclusion:
\begin{equation}
\label{H1-2}
\int_0^T \int_M \left(|u_n |^2+\| \nabla u_n \|_g^2 \right) dV(\mx) dt \leq \widetilde C_T.
\end{equation} 
\end{proof} 
Next, we need to estimate the $t$-derivative of $(u_n)$.  
\begin{lemma}
\label{Lt-der}
For any $T>0$ there exists a constant $C>0$ such that
$$
\|u_n\|_{H^1((0,T);H^{-1}(M))} \leq C.
$$
\end{lemma}
\begin{proof}
Recall that
\begin{equation}
\label{H-1}
\|u_n\|_{H^1((0,T);H^{-1}(M))}^2 = \int_0^T \|u_n(t,\mx)\|_{H^{-1}(M)}^2 dt + \int_0^T \| \partial_t u_n(t,\mx) \|_{H^{-1}(M)}^2 dt.
\end{equation} 
Since, according to Lemma \ref{L2}, we have 
$$\
\|u_n\|_{L^2((0,T);H^{-1}(M))}^2\leq \|u_n\|_{L^2((0,T);H^{1}(M))}^2 \leq C<\infty,
$$ 
it is enough to bound the second term on the right-hand side of \eqref{H-1}.

Here, using \eqref{app-12} and the continuity of the differential operator $\Div:L^2(M,TM)\to H^{-1}(M)$,
\begin{align*}
\int_0^T & \| \partial_t u_n(t,\mx) \|_{H^{-1}(M)}^2\, dt  =  \int_0^T \| -\Div  \mff_{\mx}(u_n) 
+ \Div (\Div A_{\mx}(u_n)) \|_{H^{-1}(M)} ^2\,dt\\
&\leq 2\int_0^T \left( \|\Div \mff_{\mx}(u_n)\|^2_{H^{-1}(M)} +\|\Div( \Div A_{\mx}(u_n))\|^2_{H^{-1}(M)} \right) dt\\
& \leq  \tilde C \int_0^T\left( \|\mff_{\mx}(u_n)\|^2_{L^2(M)} +\| \Div A_{\mx}(u_n)\|^2_{L^{2}(M)} \right)\, dt. 
\end{align*}
for some $\tilde{C}>0$. Finally, \eqref{bnd-1}, together with Lemma \ref{L2}, imply the boundedness of this latter expression. 
\end{proof}

Now, we can prove the existence theorem for \eqref{main-eq}, \eqref{ic}.

\begin{theorem}
\label{THM1}
The Cauchy problem \eqref{main-eq}, \eqref{ic} admits a weak solution belonging to $L^2((0,T);H^{1}(M))$ 
if $\mff \in C^{1}(M\times \R)$ and $A \in C^{2}(M\times \R)$ 
satisfy \eqref{bnd-1}.
\end{theorem}
\begin{proof} By Lemmas \ref{L2} and \ref{Lt-der} we have that $(u_n)$ is bounded in $L^2((0,T),H^1(M))$ and that
$(\pa_t u_n)$ is bounded in $L^2((0,T),H^{-1}(M))$. We may therefore apply Theorem \ref{th:aubin} (i) to the triple
$H^1(M)\subset\subset L^2(M) \subset H^{-1}(M)$ with $p=q=2$ to conclude that $(u_n)$ possesses a subsequence
(again denoted by $(u_n)$) that converges strongly in $L^2((0,T),L^2(M))$ to some $u\in L^2((0,T),L^2(M))$.

Let us show that $u$ will represent a weak solution to \eqref{main-eq}, \eqref{ic}. 
Take an arbitrary $\varphi\in C^\infty_c([0,T)\times M)$ and denote by $\varphi_n$ its projection on 
$\mathrm{span}\{e_k\mid 1\le k \le n\}$, i.e.\ (cf.\ \eqref{eq:projection_independent_of_s}), 
\[
\vphi_n(t,\mx) := \sum_{k=1}^n \langle \vphi(t,\,.\,),e_k^{(2)} \rangle_2 e_k^{(2)}(\mx) = 
\sum_{k=1}^n \langle \vphi(t,\,.\,),e_k \rangle_{L^2} e_k(\mx).
\]
We then have, taking into account that $u_n$ 
satisfies \eqref{main-eq} in the space $\mathrm{span}\{e_k\mid 1\le k \le n\}$ and using Lemma \ref{lem:trace_identity}:
\begin{align*}
0 &= -\int_0^T\int_M \big( \pa_t u_n + \Div(\mff_{\mx}(u_n)) -\Div(\Div(A_{\mx}(u_n)))\big) \vphi_n\, dV(\mx)\, dt\\
&= \int_0^T\int_M \Big( u_n \pa_t \varphi_n+ \mff_{\mx}(u_n)\cdot\nabla \varphi_n + \mathrm{tr}(A_{\mx}(u_n)\circ 
\tilde H^{\vphi_n})\Big) dV(\mx) dt\\
&\hphantom{---} + \int_M u_0(\mx) \vphi_n(0,\mx)\, dV(\mx).
\end{align*}
Consequently, for any $n\in \N$ we obtain
\begin{align*}
&\int_0^T\int_M \big(u \pa_t \varphi+ \mff_{\mx}(u)\cdot\nabla \varphi + \mathrm{tr}(A_{\mx}(u)\circ 
\tilde H^{\vphi} )\, dV(\mx)\, dt\\& + \int_M u_0(\mx) \vphi(0,\mx)\, dV(\mx) \\
&=\int_0^T\int_M \Big( (u-u_n) \pa_t \varphi+ (\mff_{\mx}(u)-\mff_{\mx}(u_n))\nabla \varphi\Big) dV(\mx) dt\\
&+ \int_0^T \int_M \mathrm{tr}( (A_{\mx}(u)-A_{\mx}(u_n)) \circ \tilde{H}^\varphi)dV(\mx) dt+\int_M u_0(\mx)(\varphi(0,\mx)-\vphi_n(0,\mx))\, dV(\mx)\\
&+\int_0^T\int_M \Big( u_n \pa_t (\varphi-\varphi_n)+ \mff_{\mx}(u_n)\nabla (\varphi-\varphi_n)\Big) dV(\mx) dt\\
&+ \int_0^T \int_M \mathrm{tr}(A_{\mx}(u_n) \circ \tilde{H}^{(\varphi-\varphi_n)})dV(\mx) dt
\end{align*}
Letting $n\to \infty$ and using $u_n\to u$ in $L^2((0,T)\times M)$ and $\varphi_n\to \varphi$ in $H^2((0,T)\times M)$, 
we conclude that $u$ is indeed a weak solution to \eqref{main-eq}, \eqref{ic}.
\end{proof}
If we additionally assume that the initial value is $C^1$, as well as stronger 
boundedness assumptions than \eqref{bnd-1} on flux and diffusion tensor, we can prove that the solution 
also increases its regularity:
\begin{theorem}\label{th:increased_regularity}
Suppose that 
\begin{equation}\label{eq:stronger_boundedness}
\begin{split}
\|A_\mx(\lambda)\|_{C^1(\R\times M)} 
+ \|\lambda A'_\mx(\lambda)\|_{L^\infty(\R\times M)} + \|A''_\mx(\lambda)&\|_{L^\infty(\R\times M)}\\
&+ \|\mff_{\mx}(\lambda) \|_{L^\infty(\R\times M)}  \leq C
\end{split}
\end{equation}
for some $C>0$. Moreover, let $u_0\in C^1(M)$. Then the
weak solution to the initial value problem \eqref{main-eq}, \eqref{ic} constructed in Theorem \ref{THM1} belongs to 
$H^{1,2}((0,T)\times M)$ (i.e., Sobolev order $1$ in $t$ and order $2$ in $\mx$).
\end{theorem}
\begin{proof}
Since $M$ is compact, it suffices to show that for any $\vphi\in C^2_c(M)$ which is compactly supported
in the domain of a chart $(U,w)$ of $M$ we have $\vphi u\in H^{1,2}([0,T]\times U)$. To reduce the
notational burden we will use the same letters for the local expressions of all involved 
quantities in the chart $(U,w)$ (i.e., $A$ for $w_*A$, $u$ for $u\circ w^{-1}$, etc.).
We may without loss of generality suppose that $\supp\vphi \subseteq K \Subset w(U)$, where
$K$ is a hypercube. A straightforward calculation shows that $\vphi u$ satisfies
\begin{equation}\label{cut}
\begin{split}
\pa_t(\vphi u) &= \vphi \pa_t u\\ 
&= \nabla^a  ({A'_{\mx}(u)^b}_a  \nabla_b(\vphi u)) - 
\nabla^a({A'_{\mx}(u)^b}_a \nabla_b\vphi\cdot  u - \vphi \nabla_b {A_\mx(\lambda)^b}_a|_{\lambda=u}) \\
&\hspace*{2em} - \nabla^a\vphi \nabla_b({A_\mx(u)^b}_a) - \nabla_a(\vphi \mff^a_\mx(u)) + \nabla_a\vphi \mff^a_\mx(u).
\end{split}
\end{equation}
Due to our assumption \eqref{strict}, this means that $\vphi u$ is the solution of a parabolic equation with
zero boundary conditions on the lateral faces of the hypercube $K$ and initial data \eqref{ic} 
multiplied by $\varphi$. Indeed, if we rewrite each of the summands in \eqref{cut} in terms of local coordinates, 
we have (using \eqref{divom})
\begin{equation}
\begin{split}
\label{(1)}
\nabla^a  ({A'_{\mx}(u)^b}_a  \nabla_b(\vphi u)) &= \Div({A'_{\mx}(u)^r}_j {g^s}_r \pa_s(\vphi u)dx^j)\\ 
&= g^{ij} \pa_i({A'_{\mx}(u)^r}_j {g^s}_r \pa_s(\vphi u)) - \Gamma^k_{il} g^{il} {A'_{\mx}(u)^r}_k {g^s}_r \pa_s(\vphi u)\\
&= \pa_i( g^{ij}{A'_{\mx}(u)^r}_j {g^s}_r \pa_s(\vphi u)) - \pa_i g^{ij} {A'_{\mx}(u)^r}_j {g^s}_r \pa_s(\vphi u)\\
& \hspace*{5em} - \Gamma^k_{il} g^{il} {A'_{\mx}(u)^r}_k {g^s}_r \pa_s(\vphi u) \\
&=:\pa_i(a^{ik} \pa_k(\varphi u)) + b^s\pa_s(\vphi u),
\end{split}
\end{equation}
with 
\begin{align*}
a^{ik} &:=g^{ij}{A'_{\mx}(u)^r}_j {g^k}_r\\
b^s &:=- (\pa_i g^{ij}+\Gamma^j_{il} g^{il}) {A'_{\mx}(u)^r}_j {g^s}_r.
\end{align*}
For the next term on the right hand side of \eqref{cut}, we have (using \eqref{divom})
\begin{align*}
&\nabla^a\big{(}{A'_{\mx}(u)^b}_a \nabla_b\vphi\cdot u - \vphi \nabla_b {A_\mx(\lambda)^b}_a|_{\lambda=u}\big{)}\\
&=
\Div\big( (u{A'_{\mx}(u)^r}_k{g^s}_r\pa_s\vphi 
-(\vphi \Div A)_k|_{\lambda=u})dx^k \big) \\
&= g^{ij}\pa_i \left( u{A'_{\mx}(u)^r}_j{g^s}_r\pa_s\vphi 
-(\vphi \Div A)_j|_{\lambda=u}
\right)
\\&\hspace*{12em}
-\Gamma_{il}^k g^{il}\left(u{A'_{\mx}(u)^r}_k{g^s}_r\pa_s\vphi 
-(\vphi \Div A)_k|_{\lambda=u}\right)
\end{align*}
Then setting
\begin{align*} 
c&:= -(\Gamma^k_{il} g^{il} +\pa_i g^{ik}) \left(u {A'_{\mx}(u)^r}_k{g^s}_r\pa_s\vphi  - (\Div A)_k|_{\lambda=u}\right)\\
c^i&:= g^{ij}\big(u {A'_{\mx}(u)^r}_j{g^s}_r\pa_s\vphi 
-\varphi  (\Div A)_j\big|_{\lambda=u}\big)
\end{align*} we conclude
\begin{equation}
\label{(2)}
\begin{split}
\nabla^a\big{(}{A'_{\mx}(u)^b}_a \nabla_b\vphi\cdot u - 
\vphi \nabla_b {A_\mx(\lambda)^b}_a|_{\lambda=u}\big{)}
= 
\pa_i c^i + c.
\end{split}
\end{equation}
Similarly, using \eqref{div11}:
\begin{align*}
&\nabla^a\vphi \nabla_b({A_\mx(u)^b}_a)+ \nabla_a(\vphi f^a_\mx(u))\\
&=g^{ks}\pa_s\vphi\big( \pa_r ({A_{\mx}(u)^r}_k) + \Gamma^r_{rl}{A_{\mx}(u)^l}_k 
-\Gamma^l_{rk}{A_{\mx}(u)^r}_l\big) 
+\pa_k(\vphi \mff_{\mx}(u)^k) + \vphi \Gamma^j_{jk} \mff_{\mx}(u)^k\\
&= \pa_r(g^{ks}\pa_s\vphi {A_{\mx}(u)^r}_k ) - {A_{\mx}(u)^r}_k \pa_r(g^{ks}\pa_s\vphi)
+ g^{ks}\pa_s\vphi\big(\Gamma^r_{rl}{A_{\mx}(u)^l}_k 
-\Gamma^l_{rk}{A_{\mx}(u)^r}_l\big)\\
&\hspace*{1em} +\pa_k(\vphi \mff_{\mx}(u)^k) + \vphi \Gamma^j_{jk} \mff_{\mx}(u)^k
\end{align*}
and from here, defining
\begin{align*}
&d^i := g^{ks}\pa_s\vphi {A_{\mx}(u)^i}_k + \vphi \mff_{\mx}(u)^i \\
&d := - {A_{\mx}(u)^r}_k \pa_r(g^{ks}\pa_s\vphi)
+ g^{ks}\pa_s\vphi\big(\Gamma^r_{rl}{A_{\mx}(u)^l}_k 
-\Gamma^l_{rk}{A_{\mx}(u)^r}_l\big) + \vphi \Gamma^j_{jk} \mff_{\mx}(u)^k,
\end{align*}we conclude
\begin{equation}
\label{(3)}
\nabla^a\vphi \nabla_b({A_\mx(u)^b}_a) + \nabla_a(\vphi \mff^a_\mx(u))=\pa_i d^i+d.
\end{equation} 
Finally, if we set 
$$
f=-c-d+\langle \grad\vphi,\mff_{\mx}(u) \rangle_g
\ \ {\rm and} \ \ f^i=-c^i-d^i,
$$ 
we see that in terms of local coordinates, \eqref{cut} satisfies (see \eqref{(1)}, \eqref{(2)}, \eqref{(3)})
\begin{equation*}
\pa_t(\varphi u)=\pa_i(a^{ik}\pa_k(\varphi u)) + b^i\pa_i(\vphi u) + \pa_i f^i + f.
\end{equation*}  
Due to \eqref{eq:stronger_boundedness} we have
$$
f,\ f^i,\ a^{ik}, \ b^i \in L^\infty([0,T)\times K).
$$ 
Using the analog of \eqref{H1-1} for $u$, it follows that $(x,t)\mapsto (\vphi u)(t,x)\in V^{1,0}_2([0,T]\times U)$
(in the notation of \cite{LSU}). We may therefore apply \cite[Ch.\ III, Th.\ 7.1 and Cor.\ 7.1]{LSU},
which, due to Theorem \ref{THM1}, \eqref{strict} and the fact that we have zero boundary conditions,
leads to
\begin{equation}
\label{interp1}
\| \varphi u \|_{L^\infty([0,T]\times U)} \leq C <\infty
\end{equation} From here and \cite[Ch.\ III, Th. 10.1]{LSU}, we obtain that for some $\alpha>0$
\begin{equation}
\label{regularity}
\varphi u \in C^{\alpha}((0,T)\times U'), \ \ U'\Subset U.
\end{equation} Thus, by taking $U'$ on which $\varphi \equiv 1$, we see that $u$ satisfies the following boundary value problem for a quasi-linear equation on $[0,T]\times U'$ (we denote $S_T=(0,T)\times \pa U'$):
\begin{align*}
\pa_t  u &=\pa_i(a^{ik}\pa_k u) + b^i\pa_iu + \pa_i f^i + f \\
u|_{S_T}&= \gamma_{S_T}(u) \in C^{\alpha}(S_T), \\
 u|_{t=0}&=u_0
\end{align*} where $\gamma_{S_T}$ is the trace operator. 

According to \cite[Ch.\ V, Th.\ 6.4]{LSU}, our regularity assumptions on the initial data, and again invoking \eqref{eq:stronger_boundedness}, 
we conclude that $u\in H^{1,2}((0,T)\times K')$ for any $K'\subset\subset K$. Since we have chosen an arbitrary chart and the manifold $M$ is compact, we conclude that $u\in H^{1,2}((0,T)\times M)$.
\end{proof} Uniqueness of such weak solutions is not automatic, unless we assume additional properties of the coefficients as we shall see in the next section.

\begin{remark}
\label{LB}
We remark that the same method can be used if we add the Laplace-Beltrami operator $\Delta u$ or even the 
bi-harmonic operator $\Delta^2 u$ on the right hand side of \eqref{main-eq}, in which case we merely need to require 
semi-definiteness of the tensor $A'_{\mx}$. In other words, we may consider the equation 
\begin{equation}
\label{non-strict}
\pa_t u +\Div \mff_{\mx}(u)=\Div(\Div ( A_{\mx}(u) ) )+\epsilon \Delta u, \ \ \mx \in M
\end{equation} 
for a positive parameter $\eps>0$ and a $(1,1)$-tensor $\mx\mapsto A_\mx(\lambda)\in \mathcal{T}^1_1(M)$ that satisfies:
\begin{equation}
\label{non-strict-cond}
\langle  A'_\mx(\lambda) \mxi, \mxi \rangle \geq 0,
\end{equation} 
and supplement it with the initial conditions \eqref{ic}. We can then apply the same method as for 
problem \eqref{main-eq}, \eqref{ic}. Indeed, in the case of equation \eqref{non-strict}, the system of 
ODEs obtained after inserting the approximation \eqref{app-1} has the form
\begin{equation}
\label{ns-app-2}
\begin{split}
\dot{\alpha}^n_j =\int_M & \langle\mff_{\mx}(u_n),\nabla e_j(\mx)\rangle \,dV(\mx)+
\int_M  \mathrm{tr} (A_{\mx}(u_n) \circ \tilde{H}^{ e_j})(\mx) dV(\mx)\\ 
&+\eps\lambda_j \int_M e_j^2(\mx) d\mx.
\end{split}
\end{equation} Then, Lemma \ref{L1} remains valid and the rest of the proof is the same since \eqref{non-strict} is 
a strictly parabolic PDE. We note that the estimates given in Lemma \ref{L2} and Lemma \ref{Lt-der} are based 
on the strict parabolicity of the equation, and that the existence proof in turn is based on those Lemmas.
{Alternatively, we may also reduce
directly to the situation studied previously by defining a new $(1,1)$-tensor field 
$\tilde A_\mx(\lambda) := A_\mx(\lambda) + \eps\lambda I$, where $I$ is the identity tensor, $I={\delta^a}_b$. Indeed, then $\Div(\Div ( \tilde A_{\mx}(u) ) )
= \Div(\Div ( A_{\mx}(u) ) ) + \eps \Delta u$.} 

Moreover, we note that the method of proof of Theorem \ref{THM1} applies on question of $L^2((0,T);H^1(M))$ convergence of any sequence of approximate solutions to \eqref{main-eq} satisfying bounds from Lemma \ref{L2} and Lemma \ref{Lt-der}.
\end{remark}

\section{The case of unbounded flux and diffusion}\label{sec:unbounded_flux}

In the previous section we proved existence of a global solution to \eqref{main-eq}, \eqref{ic} 
under the assumption that the flux $\mff_{\mx}$ and the diffusion $A_{\mx}$ are bounded 
in the sense of \eqref{bnd-1}. Here, we shall show that the initial value problem \eqref{main-eq}, \eqref{ic} 
also has a global solution also in the absence of \eqref{bnd-1}, 
when imposing additional assumptions that induce a suitable maximum principle.

More precisely, we suppose
\begin{itemize}

\item The initial condition satisfies $0 \leq u_0 \leq 1$;

\item The geometry compatibility condition \cite{BLf,GKM} is satisfied 
\begin{equation}
\label{geomcomp}
\Div \mff_{\mx}(\lambda)=\Div (\Div (A_{\mx}(\lambda)) \ \ \text{for every $\lambda\in \R$}.
\end{equation}

\end{itemize} Condition \eqref{geomcomp} means that the divergence of the (diffusive) flux $\mff_{\mx}(\lambda)-\Div (A_{\mx}(\lambda))$ is zero. If we model the dynamics of a fluid in the Euclidean case, the latter means that the fluid is incompressible. Indeed, if we denote by $\rho$ the density of the fluid, then its change in a control volume is given by:
\begin{equation}
\label{1}
\frac{D\rho }{Dt}={\Div}(a_\mx(\rho)\cdot\nabla \rho), \ \ a_\mx(\lambda) =\pa_\lambda A(\mx,\lambda)  
\end{equation}i.e. the change of the density occurs only due to the diffusion effects. In \eqref{1} we use the standard convention in the frame of which $\frac{D\rho }{Dt}=\frac{\pa \rho}{\pa t}+\frac{d\mx}{dt} \cdot \nabla \rho$ is the material derivative for the flow velocity $\frac{d\mx}{dt}=(\frac{dx_1}{dt},\dots,\frac{dx_d}{dt})$. 
Under the assumption that we are in the Euclidean situation, equation \eqref{main-eq} can be rewritten as
\begin{equation}
\begin{split}
\frac{\pa \rho}{\pa t}+\pa_\lambda \big(\mff_\mx(\lambda)-{\Div}A_\mx(\lambda)\big)\big{|}_{\lambda=\rho} \cdot \nabla \rho & +  {\rm Div} (\mff_\mx(\lambda)-{\Div}A_\mx(\lambda))\big{|}_{\lambda=\rho}\\
&={\Div} (a_\mx(\rho) \cdot \nabla \rho)
\label{2}
\end{split}
\end{equation} Then, since the velocity of the fluid point is
$
\frac{d\mx}{dt}=\pa_\lambda \big(\mff_\mx(\xi)-{\Div}A_\mx(\xi)\big)\big{|}_{\xi=\rho}
$ we get by subtracting \eqref{2} from \eqref{1}:
$$
(\Div \mff_\mx(\lambda)-\Div ( \Div ( A_\mx(\lambda))))\big{|}_{\lambda=\rho}=0,
$$ which is the geometry compatibility condition (since the latter must hold for any possible density $\rho$).

The numerical values in the condition $0 \leq u_0 \leq 1$ are not essential, 
but often are a natural choice as the unknown function may describe, for instance, the concentration 
of fluids in porous media. We refer to \cite{GKM}, where this assumption is imposed and 
where additional context is provided. In particular, we will use the following
result (a modification of \cite[Th.\ 1]{GKM}):
\begin{theorem}\label{th:GKM_theorem1}
	\label{thm1}Assume that the geometry compatibility condition \eqref{geomcomp} holds and that 
	$u: \R^+\times M \to \R$ is a $H^{1,2}([0,T]\times M)$, $T>0$, solution to 
	\eqref{non-strict}. Then for any convex function $S\in C^2(\R)$ such that $S(0)=0$ we have
	\begin{align}\label{a1}
	\begin{split}
	&\pa_t S(u)+ \Div \int_0^{u(t,\mx)} \mff'_{\mx}(\xi)S'(\xi)\, d\xi=\Div \Div \Big(\int_0^{u(t,\mx)} A_{\mx}'(\xi)S'(\xi)\, d\xi\Big)\\
	&+\epsilon \Delta S(u)-\eps S''(u) |\nabla u|^2 -S''(u) \langle A'_{\mx}(u) \nabla u,\nabla u \rangle,
	\end{split}
	\end{align} 
	where $\mff'=\pa_\xi \mff$ and $A'=\pa_\xi A$.
\end{theorem}
\begin{proof}
By multiplying \eqref{non-strict} by $S'(u)$ 
we get after applying the exact same procedure as in the proof of  \cite[Th.\ 1]{GKM}:
\begin{equation}
\label{deg-par}
\begin{split}
    & 0= S'(u)\pa_t u+ S'(u) \Div \mff_{\mx}(u)-S'(u)\Div \Div A_{\mx}(u) -\eps S'(u)\Delta u\\
	&=\pa_t S(u)+ \Div \int_0^{u(t,\mx)} \mff'_{\mx}(\xi)S'(\xi)\, d\xi-\Div \Div \Big(\int_0^{u(t,\mx)} A_{\mx}'(\xi)S'(\xi)\, d\xi\Big)\\
	&+S''(u) \langle A'_{\mx}(u) \nabla u,\nabla u \rangle -\eps S'(u)\Delta u.
\end{split}
\end{equation} As for the Laplace-Beltrami operator, we have
\begin{equation}
\label{lb-1}
S'(u) \Delta u=\Delta S(u)-S''(u) |\nabla u|^2.
\end{equation}
Combining \eqref{deg-par} and \eqref{lb-1}, the claim follows.
\end{proof}

\begin{remark} In \cite{GKM}, this result is proved supposing that $u$ be bounded 
and non-negative, but an inspection of the proof shows that these assumptions are not required. 
\end{remark}

\begin{theorem} Let $\mff \in C^{2}(M\times \R)$ and $A \in C^{2}(M\times \R)$. Then 
under the assumptions $u_0\in C^1(M)$, $0 \leq u_0 \leq 1$ and \eqref{geomcomp}, the Cauchy problem \eqref{main-eq}, 
\eqref{ic} admits a unique bounded weak solution belonging to $H^{(1,2)}((0,T)\times M)$.
\end{theorem}
\begin{proof} To begin with, we consider an alternative problem to \eqref{main-eq}, where $\mff$ and $A$ 
are replaced by suitably truncated versions $\tilde \mff$ and $\tilde A$, respectively, satisfying the global bounds \eqref{eq:stronger_boundedness}.
Concretely, set 
\begin{equation}\label{tilmf}
\tilde{\mff}_{\mx}(\lambda):=\mff_{\mx}(\chi(\lambda))
\end{equation}
and 
\begin{equation}\label{tilA}
\tilde{A}_{\mx}(\lambda)=:A_{\mx}(\chi(\lambda)),
\end{equation}
where $\chi :\R \to \R$ is a smooth function such that $\chi(\lambda)=\lambda $ for $\lambda \in [0,1]$, $\chi $ is constant on $(-\infty,-1]$ and on $[2,\infty ) $ and such that $\chi'\geq 0$ everywhere. 

We consider the following equation
\begin{equation}
\label{appr-1}
\pa_t \tilde{u}_\eps +\Div \tilde\mff_{\mx}(\tilde{u}_\eps)=\Div(\Div ( \tilde{A}_{\mx}(\tilde{u}_\eps) ) )+\eps \Delta \tilde{u}_\eps.
\end{equation} 
Then \eqref{eq:stronger_boundedness}
is satisfied for $\tilde\mff \in C^2$ and $\tilde{A}\in C^2$, and the geometry compatibility condition
\eqref{geomcomp} holds. By Theorem \ref{th:increased_regularity} and Remark \ref{LB}, the initial value problem corresponding to \eqref{appr-1} 
with $0\le u_0\le 1$ has a solution $\tilde{u}_\eps\in H^{1,2}((0,T)\times M))$.
We may now apply Theorem \ref{th:GKM_theorem1} to $\tilde{u}_\eps$. More precisely, we insert for $S$ as limiting
cases for the regularity (cf.\ \cite[Sec.\ 2]{[CP]}) the semi-entropies
$$
\eta_+(\tilde{u}_\eps)=|\tilde{u}_\eps-1|_+=\begin{cases} 
\tilde{u}_\eps-1, &\tilde{u}_\eps-1>0\\
0, & \text{otherwise}
\end{cases}  \quad  \eta_-(\tilde{u}_\eps)=|\tilde{u}_\eps|_-=\begin{cases} 
|\tilde{u}_\eps|, &\tilde{u}_\eps<0\\
0, & \text{otherwise}
\end{cases}
$$ 
Then, integrating \eqref{a1} over $[0,T)\times M$, we get, taking into account the positive definiteness of 
$(\tilde A^\eps)'_\mx(\lambda) + \eps  I$,
\begin{equation}
\label{max-princ}\begin{split}
&\int_M |\tilde{u}_\eps(T,\mx)-1|_+ \,dV(\mx) \leq \int_M |u_0(\mx)-1|_+ \,dV(\mx)=0, \ \ {\rm and} \\ 
&\int_M |\tilde{u}_\eps(T,\mx)-0|_- \,dV(\mx) \leq \int_M |u_0(\mx)-0|_- \,dV(\mx)=0,
\end{split}
\end{equation}
implying that $0\leq \tilde{u}_\eps \leq 1$. Thus $\tilde{u}_\eps$ never leaves the $\lambda$-region in which, according to \eqref{tilmf}, \eqref{tilA}, and the choice of $\chi $, $\tilde{\mff}_{\mx}\equiv \mff_{\mx}$ 
and $\tilde{A}_{\mx}\equiv A_{\mx}$.

Thus,  in the range of $u_\eps$, the equation \eqref{appr-1} is strictly parabolic with the parabolicity constant $c$ from \eqref{strict} independent of $\eps$. Therefore, repeating the proof of Theorem \ref{THM1}, we conclude that the sequence $(\tilde u_\eps)$ converges in $L^2((0,T);H^1(M))$ toward a function $\tilde{u}$ 
solving the problem
\begin{equation}
\label{appr-2}
\begin{split}
&\pa_t \tilde{u} +\Div \mff_{\mx}(\tilde{u})=\Div(\Div ( {A}_{\mx}(\tilde{u}) ) )\\
&u|_{t=0}=u_0(\mx).
\end{split}
\end{equation} Hence $\tilde u$ is a solution to the original initial value problem \eqref{main-eq}, \eqref{ic}. From here and since the conditions of Theorem \ref{th:increased_regularity} are satisfied on the range of $\tilde{u}$, we conclude that the solution constructed above is in $H^{1,2}((0,T)\times M)$ and bounded. Consequently, it is an entropy admissible solution in the sense of \cite[Def.\ 3]{GKM} (due to \cite[Th.\ 1 and Th.\ 2]{GKM}).
It is therefore unique by \cite[Cor.\ 14]{GKM}.  
\end{proof}

\begin{remark}
The assumption $\mff \in  C^{2}(M\times \R)$ can be relaxed by involving mollification with respect to $\lambda \in \R$ into the construction. The mollification does not affect the geometry compatibility conditions and thus it does not affect the proof of the previous theorem. However, in this case we would have several additional technical steps which blur the ideas of the proof. Let us just briefly explain the main steps.

First, we mollify both $\mff_{\mx}(\lambda)$ and $A_{\mx}(\lambda)$ with respect to $\lambda \in \R$ via the standard convolution kernel, say $\rho_\eps$. This does not affect the geometry compatibility condition. For this new flux $\mff_{\mx}^\eps(\lambda)$ and diffusion $A_{\mx}^\eps(\lambda)$ we solve \eqref{main-eq}, \eqref{ic}, 
obtaining a family of bounded functions $(u_\eps)$ satisfying
\begin{align*}
\pa_t u_\eps +\Div \mff^\eps_{\mx}(u_\eps)&=\Div(\Div ( A^\eps_{\mx}(u_\eps) ) ), \ \ \mx \in M,\\
u_\eps|_{t=0} &=u_0(\mx) \in L^\infty(M).
\end{align*} 
This family $(u_\eps)$ satisfies the estimates from Lemma \ref{L2} and Lemma \ref{Lt-der} and therefore is strongly precompact in $L^2([0,T]\times M)$. A limiting function $u$ (along a subsequence of $(u_\eps)$) then is the solution of \eqref{main-eq}, \eqref{ic}.
\end{remark}


\begin{remark}\label{rem:general_Q}
We note that we can also consider a general parabolic equation of the form (see \cite{MM})
\begin{equation}
\label{general}
\pa_t u =Q[u] \ \ {\rm on} \ \ (0,T)\times M,
\end{equation} again with the initial data \eqref{ic} on a $d$-dimensional compact orientable Riemannian manifold. 
The operator $Q$ has the following local expression:
\begin{align*}
Q[u](t,\mx) = &\pa_{j_1\dots j_{p}} \left( A^{i_1\dots i_{p} j_1\dots j_p}(t,\mx,\nabla u, \dots, \nabla^{p-1} u)\pa_{i_1\dots i_{p}}u(t,\mx)\right)\\
&+b(t,\mx,u,\nabla u, \dots, \nabla^{p-1} u),
\end{align*} 
where the tensor $A=(A^{i_1\dots i_{p} j_1\dots j_p})$ is a locally elliptic smooth $(2p,0)$-tensor satisfying uniformly with respect to all the variables
$$
-(-1)^{p}\sum\limits_{i_1,\dots,i_p,j_1,\dots,j_p} A^{i_1 \dots i_{p} j_1\dots  j_p} \xi_{i_1\dots i_p} \xi_{j_1\dots j_p} \geq c \sum\limits_{i_1,\dots,i_p} |\xi_{i_1\dots i_p}|^2
$$ for some $c>0$ and every $\mxi=(\xi_{i_1\dots i_p})_{i_j\in \{1,\dots,d\}}$.

The Galerkin approximation from Section \ref{sec:galerkin} will induce a system of equations (corresponding to \eqref{app-2}):
\begin{equation}
\label{gen-1}
\dot{\alpha}^n_j=\int_M Q[u_n](t,\mx) e_j(\mx) dV(\mx)=F(t,\malpha_n), \ \ j=1,\dots,n,
\end{equation} 
where $\malpha_n=(\alpha_1^n,\dots,\alpha_n^n)$. Moreover, under appropriate growth assumptions on $A$ and $b$, 
repeating the methods from Lemma \ref{L2} and Lemma \ref{Lt-der}, parabolicity of the operator $Q$ provides 
uniform apriori boundedness of $(u_n)$ in  $L^2((0,T); H^p(M)) \cap H^{1}((0,T);H^{-p}(M))$ which in turn provides strong precompactness of the sequence of approximate solutions (again by means of the Aubin-Lions-Simon lemma).

\end{remark}

\section{Stochastic parabolic differential equation}\label{sec:stochastic}

Let us first introduce the notions that we need. By $(\Omega, {\cal F}, ({\cal F}_t), {\bf P})$ we denote the stochastic basis with the sample space $\Omega$, the $\sigma$-algebra ${\cal F}$, the natural filtration ${\cal F}_t$ generated by the Wiener process $W_t$, and the probability measure ${\bf P}$. 

The Wiener process $W_t$ is a stochastic process which has independent Gaussian increments in the sense that $W_{t+u}-W_t$ is independent of the past values $W_s$, $s<t$, and $W_{t+u}-W_t \sim {\cal N}(0,u)$. Finally, $W_t$ has continuous paths i.e.\ with probability $1$, $W_t$ is continuous in $t$.

The stochastic forcing $dW_t$ is the It\^o integral defined for appropriate functions $g$ \cite[Definition 3.4.1]{Oks} as the following limit in probability:
$$
\int_0^T g(t,\omega) dW_t=\lim\limits_{n\to \infty} \sum\limits_{k=0}^{n}g(t_k,\omega) (W(t_{k+1},\omega)-W(t_{k},\omega)).
$$ We see that the latter integral is actually a random variable with respect to the filtration ${\cal F}_t$. 

We consider the following spaces (not distinguishing between functions {and the corresponding equivalence classes)}:
\begin{align*}
&L^2_{{\bf P}}(\Omega; L^2((0,T); H^1(M)))\\&=\{u: (0,T)\times M \times \Omega \to \R: \, \int_\Omega \int_0^T \|u(t,\cdot,\omega)\|^2_{H^1(M)} dt d{\bf P}(\omega)<\infty  \}\\
&L^2_{{\bf P}}(\Omega; C^{1/2}((0,T); H^{-1}(M))):=\\
&\{u: (0,T)\times M \times \Omega \to \R: \, \int_\Omega \sup\limits_{\Delta t>0} \frac{\|u(t+\Delta t,\cdot,\omega)-u(t,\cdot,\omega)\|^2_{H^{-1}(M)}}{\Delta t} d{\bf P}(\omega)<\infty  \}\\
&L^2_{{\bf P}}(\Omega; L^2(M)))=\{u: M \times \Omega \to \R: \, \int_\Omega \int_M |u(\mx,\omega)|^2 dV(\mx) d{\bf P}(\omega)<\infty  \}.
\end{align*}

%
%

 We assume the following for the coefficients of \eqref{main-eq-stoch}:

\begin{itemize}

\item[(i)] $\mff \in C^1(M\times \R)$ is such that $\sup\limits_{\lambda\in \R}\|\mff(\cdot,\lambda)\|_{L^2(M)}<C$ and $\| \mff\|_{C^1(M\times \R)}<C$, and, for simplicity, $\mff(\mx,0) \equiv 0$;

\item[(ii)] For every $\mx \in M$, the map $\lambda \mapsto A_{\mx}(\lambda)$ is linear;

\item[(iii)] $\Phi\in C^1_0(M\times \R)$. 

\end{itemize}

We now supplement equation \eqref{main-eq-stoch} with the initial condition
\begin{equation}
\label{ic-stoch} 
u|_{t=0}=u_0(\mx,\omega)=\sum\limits_{k=1}^\infty \alpha_k(\omega) e_k(\mx), \ \ (\mx,\omega)\in M \times \Omega,
\end{equation} where $u_0\in L^\infty_{{\bf P}}(\Omega; L^2(M)))$ in the sequel if not stated otherwise. Note that this implies 
\begin{equation}
\label{finite}
 \sum\limits_{k=1}^\infty  |\alpha_k(\omega)|^2  \leq C <\infty \ \ \text{almost surely}.
\end{equation}

We shall prove the following theorem.
\begin{theorem}
\label{thm-main-st}
If the mapping $\lambda \mapsto A_{\mx}(\lambda)$ is linear for every $\mx \in M$, then the Galerkin approximation:
\begin{equation}
\label{un}
u_n(t,\mx,\omega)=\sum\limits_{k=1}^n \alpha^n_k (t,\omega) e_k(\mx), \ \ t\in [0,T), \ \ \mx\in M, \ \ \omega\in \Omega,
\end{equation} converges in $L^2_{{\bf P}}(\Omega; L^2((0,T); H^1(M)))$ towards $u\in L^2_{{\bf P}}(\Omega; L^2((0,T); H^2(M))) \cap  L^2_{{\bf P}}(\Omega; C^{1/2}((0,T); L^2(M)))$ representing a solution to \eqref{main-eq-stoch}, \eqref{ic-stoch} in the sense that for every $\varphi \in H^2(M)$ and every $\Delta t>0$ it holds  
\begin{equation}\label{eq:stoch_equation}
\begin{split}
&\int_M \left( u(t+\Delta t,\mx,{\omega})-u(t,\mx,{\omega}) \right) \varphi(\mx)\,dV(\mx)= 
\int_t^{t+\Delta t} \int_M \mff(\mx,u)\nabla \varphi(\mx) \, dV(\mx) dt\\
&-\int_t^{t+\Delta t} \int_M {\Div(A_{\mx}\cdot u)(\nabla\vphi)} dV(\mx) dt+\int_t^{t+\Delta t} \int_M \Phi (\mx,u) \varphi(\mx) \,dV(\mx) dW_t.
\end{split}
\end{equation}
\end{theorem}

In order to prove {Theorem \ref{thm-main-st}}, we shall need {two fundamental results from the It\^o calculus},
the first of which is (see \cite[Th.\ 4.1.2]{Oks},\cite[Th.~5.9]{Deck}):
\begin{lemma}
\label{IL1} (It\^o lemma)
Let $X_t$ be {an It\^o process given by}
\begin{equation}
\label{StocProc}
dX_t = \mu_1 dt + \sigma_1 dW_t.  
\end{equation}  For each {$f\in C^{1,2}([a,b]\times \R)$, $f=f(t,z)$, also $f(t,X_t)$ is an It\^o process, and}
\begin{align}
\label{ItoFormula}
df({t},X_t)&=\left( \frac{\pa f}{\pa t} + \mu_1 \frac{\pa f}{\pa z} + \frac{\sigma^2_1}{2} \frac{\pa^2 f}{\pa z^2} \right) dt + \sigma_1 \frac{\pa f}{\pa z} dW_t
\end{align}
holds.
\end{lemma} We remark here that {\eqref{StocProc}} is actually an informal way of expressing the integral equality
\begin{equation}
\label{integral}
X_{t_0+s}-X_{t_0}=\int_{t_0}^{t_0+s}\mu_1 dt + \int_{t_0}^{t_0+s} \sigma_1 dW_t, \ \ \forall t_0,s >0.  
\end{equation}
{To formulate the second prerequisite, as in \cite[Def.\ 3.1.4]{Oks}, by $\mathcal{V}(S,T)$ we denote the 
set of all $f:[0,\infty)\times\Omega \to \R$ that are measurable, $\mathcal{F}_t$-adapted and satisfy $E\big[ \int_S^T f(t,\omega)^2\,dt \big]<\infty$.
Then by \cite[Cor.\ 3.1.7]{Oks}: 
\begin{lemma}\label{IL2}  (It\^o isometry) For any $f\in \mathcal{V}(S,T)$
\begin{equation*}
E\Big[ \Big( \int_S^T f(t,\omega) dW_t \Big)^2\Big]=E\big[ \int_S^T f(s,\omega)^2 dt \Big].
\end{equation*}
\end{lemma}
}

\subsection{Proof of Theorem \ref{thm-main-st}} 
 
We seek functions $\alpha_k$, $k=1,\dots,n$ such that  \eqref{main-eq-stoch} is satisfied in the following sense (cf.\ \eqref{app-1}): Let 
\begin{equation}
\label{app-1-stoch}
u_n(t,\mx,\omega)=\sum\limits_{k=1}^n \alpha^n_k(t,\omega) e_k(\mx).
\end{equation} 
Then we require that almost surely for any $\varphi\in \mathrm{span}\{e_k\}_{k=1,\dots,n}$ we have:
\begin{equation}
\label{app-1-s}
\begin{split}
\int_M d u_n \varphi dV(\mx)&= \int_M \mff(\mx,u_n)\nabla \varphi \, dV(\mx) dt\\
& {-\int_M \Div(A_{\mx}\cdot u_n)(\nabla\vphi)\, dV(\mx) dt}
+\int_M \Phi (\mx,u_n) \varphi dV(\mx) dW_t.
\end{split}
\end{equation} If we put here $\varphi=e_j$, $j=1,\dots, n$, using orthogonality of $\{e_k\}_{k\in \N}$, we get the following system of stochastic ODEs:
\begin{equation}
\label{system}
\begin{split}
d {{\alpha}_j^n}&=\int_M \mff(\mx,u_n)\nabla e_j(\mx) \, dV(\mx) dt\\
& {-\int_M \Div(A_{\mx}\cdot u_n)(\nabla e_j)\, dV(\mx) dt}
+\int_M \Phi (\mx,u_n) e_j(\mx) dV(\mx) dW_t.
\end{split}
\end{equation} According to \eqref{bnd-1} and since $\Phi\in C^1_0(M \times \R)$, from \cite[Theorem 5.2.1]{Oks} we infer that \eqref{system} satisfying finite initial data $\alpha_{j0}(\omega)$ (in the sense that $\alpha_{j0}(\omega)\leq C<\infty$ almost surely) has
{a unique}
globally defined {$t$-continuous} solution. In particular, if 
$u_0(\mx,\omega)=\sum\limits_{j\in \N} \alpha_{j0}(\omega) e_j(\mx)$, and we choose
\begin{equation}
\label{id-fin}
\alpha_j(0,\omega)=\alpha_{j0}(\omega), 
\end{equation} 
we obtain a sequence $(u_n)$ satisfying for every $n\in \N$ the relation \eqref{system} (see \eqref{finite}). 

\subsection*{$H^s$-estimates}

Let us now derive $H^s$-estimates, $s=1,2$, for $(u_n)$. To this end, from \eqref{system} and the It\^o lemma
we infer:
\begin{align}
\label{basic-stoch}
\frac{1}{2}d |{{\alpha}^n_j}|^2&=\int_M \mff(\mx,u_n) {(\nabla e_j) \alpha_j^n} \, dV(\mx) dt\\
&- \int_M \Div( A_\mx \cdot u_n){((\nabla e_j) \alpha_j^n)} \, dV(\mx) dt +
{\frac{1}{2}}\big[ \int_M  \Phi (\mx,u_n) e_j(\mx) dV(\mx) \big]^2 dt \nonumber  
\\&+\int_M \alpha_j^n \Phi (\mx,u_n) e_j(\mx) dV(\mx) dW_t.
\nonumber
\end{align} By summing the latter relation for $j=1,\dots,n$, we conclude from the {fact that the $e_j$ 
form a complete orthonormal system in $L^2(M)$:}
\begin{align*}
\frac{1}{2}d \|u_n(t,\cdot,\omega)\|^2_{L^2(M)}&\leq \int_M \mff(\mx,u_n)\nabla u_n(t,\mx,\omega) \, dV(\mx) dt\\
&- \int_M \Div( A_\mx \cdot u_n)( \nabla u_n(t,\mx,\omega) )\, dV(\mx) dt + \frac{1}{2}\| \Phi (\mx,u_n)\|^2_{L^2(M)}  dt  
\\&+\int_M  \Phi (\mx,u_n) u_n(t,\mx,\omega) dV(\mx) dW_t.
\end{align*} 
From here, integrating the relation over $(0,T)$ and using Lemma \ref{Lpar} (i'), we have the following estimate 
\begin{equation}
\label{ISCL-21}
\begin{split}
&\int_M \frac{  |u_n(T,\mx,\omega)|^2}{2} dV(\mx) + c  \int_0^T \int_M \|\nabla u_n\|_g^2 dV(\mx)\\& 
\leq \int_M \frac{  |u_0(\mx,\omega)|^2}{2} dV(\mx)+\int_0^T (\|\mff(\cdot,u_n)\|^2_{L^2(M)} +\| \Phi (\mx,u_n)\|^2_{L^2(M)})  dt  \\&\ \ +C \int_0^T \int_M | u_n |^2 dV(\mx) dt +  \int_0^T\int_M \Phi(\mx,u_n) u_n dV(\mx) dW_t.
\end{split}
\end{equation} 
This implies
\begin{equation*}
\begin{split}
&E\big(\int_M \frac{  |u_n(T,\mx,\cdot)|^2}{2} dV(\mx) \big) \\& 
\leq E\big( \int_M \frac{  |u_0(\mx,\cdot)|^2}{2} dV(\mx)\big)+ T (\|\sup\limits_{\lambda}|\mff(\cdot,\lambda)|\|^2_{L^2(M)} +\| \sup\limits_{\lambda}| \Phi (\mx,\lambda)| \|^2_{L^2(M)})  \\&\ \ +CE\big( \int_0^T \int_M | u_n |^2 dV(\mx) dt\big) .
\end{split}
\end{equation*} Then, using the Gronwall inequality, we get
\begin{equation*}
\begin{split}
E\big( \int_M \frac{  |u_n(T,\mx,\cdot)|^2}{2} dV(\mx) \big)& \leq e^{CT} \Big(E\big( \int_M \frac{  |u_0(\mx,\cdot)|^2}{2} dV(\mx)\big)\\&+ T (\|\sup\limits_{\lambda}|\mff(\cdot,\lambda)|\|^2_{L^2(M)} +\| \sup\limits_{\lambda}| \Phi (\cdot,\lambda)| \|^2_{L^2(M)}) \Big).  
\end{split}
\end{equation*} 
Combining this with \eqref{ISCL-21}, we finally have
\begin{align}
\label{H1}
&E\big( \int_M \frac{  |u_n(T,\mx,\omega)|^2}{2} dV(\mx) + c  \int_0^T \int_M \|\nabla u_n\|_g^2 dV(\mx) \Big)\\& 
\leq E\big( \int_M \frac{  |u_0(\mx,\cdot)|^2}{2} dV(\mx)\big)+ T (\|\sup\limits_{\lambda}|\mff(\cdot,\lambda)|\|^2_{L^2(M)} +\| \sup\limits_{\lambda}| \Phi (\mx,\lambda)| \|^2_{L^2(M)}) \nonumber \\ 
&+CT e^{CT} \Big(E\big( \int_M \frac{  |u_0(\mx,\cdot)|^2}{2} dV(\mx)\big)+ T (\|\sup\limits_{\lambda}|\mff(\cdot,\lambda)|\|^2_{L^2(M)} +\| \sup\limits_{\lambda}| \Phi (\mx,\lambda)| \|^2_{L^2(M)}) \Big).
\nonumber
\end{align}
To prove higher order estimates, we multiply \eqref{basic-stoch} by $\lambda_j^2$ and then sum the obtained expressions. Using \eqref{eq:e_j_sobolev_products}, we get: 
\begin{align*}
\frac{1}{2}d \| u_n(t,\cdot,\omega)\|^2_{H^1(M)}=&-\int_M \Div(\mff(\mx,u_n)) \, 
\Big( \sum\limits_{j=1}^n \lambda^2_j\alpha_j(t,\omega) e_j(\mx)\Big) \, dV(\mx) dt\\
&- \int_M \Div( A_\mx \cdot u_n)\Big( \nabla \Big(\sum\limits_{j=1}^n \lambda^2_j\alpha_j(t,\omega) e_j(\mx)\Big)\Big)\, dV(\mx) dt \\
&+ {\frac{1}{2}}\sum\limits_{j=1}^n \lambda^2_j [\int_M \Phi (\mx,u_n)e_j(\mx) dV(\mx)\big]^2 dt  
\\&+\int_M  \Phi (\mx,u_n) \sum\limits_{j=1}^n \lambda^2_j \alpha_j(t,\omega) e_j(\mx) dV(\mx) dW_t.
\end{align*}
Next, since $(I-\Delta)e_j=\lambda^2_j e_j$ we can rewrite the latter in the form
\begin{align*}
\frac{1}{2}d \| u_n(t,\cdot,\omega)\|^2_{H^1(M)}&=-\int_M \Div(\mff(\mx,u_n)) \, 
\Big( \sum\limits_{j=1}^n \lambda^2_j\alpha_j(t,\omega) e_j(\mx)\Big) \, dV(\mx) dt\\
&- \int_M \Div( A_\mx \cdot u_n)(\nabla (u_n - \Delta u_n)) \, dV(\mx) dt \\
&+{\frac{1}{2}}\sum\limits_{j=1}^n \lambda^2_j [\int_M \Phi (\mx,u_n)e_j(\mx) dV(\mx) \big]^2 dt  
\\&+\int_M  \Phi (\mx,u_n) \sum\limits_{j=1}^n \lambda^2_j \alpha_j(t,\omega) e_j(\mx) dV(\mx) dW_t.
\end{align*} 
Using the Peter-Paul inequality and again \eqref{eq:e_j_sobolev_products}, 
we have for appropriate constants $\tilde{K}$ and $k$:
\begin{align*}
\frac{1}{2}d \| u_n(t,\cdot,\omega)\|^2_{H^1(M)}& 
\leq  \tilde{K}\|\mff(\cdot,u_n)\|^2_{H^1(M)}dt+ 
\frac{k}{4}\|u_n(t,\cdot,\omega) \|^2_{H^2(M)} dt
\\
&- \int_M \Div( A_\mx \cdot u_n)(\nabla u_n )\, dV(\mx) dt 
\\&- \int_M \Div \Div(A_\mx \cdot u_n) \, \Delta u_n \, dV(\mx) dt
\\& + {\frac{1}{2}}\sum\limits_{j=1}^n \lambda^2_j [\int_M \Phi (\mx,u_n)e_j(\mx) dV(\mx) \big]^2 dt  
\\&+\int_M  \Phi (\mx,u_n) \sum\limits_{j=1}^n \lambda^2_j \alpha_j(t,\omega) e_j(\mx) dV(\mx) dW_t.
\end{align*} 
Now, we take the expectation of the latter estimate. {Then the last term vanishes and using} Lemma \ref{Lpar} (ii) we conclude
\begin{align*}
& \frac{1}{2} E\big( \| u_n(t,\cdot,\cdot)\|^2_{H^1(M)}\big)+ \tilde{k} E\big(\int_0^t \| u_n(t',\cdot,\cdot)\|^2_{H^2(M)} dt' \big) 
\\
 &\leq \frac{1}{2} E\big( \| u_{n0}(\cdot,\cdot)\|^2_{H^1(M)}\big)+ \tilde{K} E\big( \int_0^t  \|\mff(\cdot,u_n)\|^2_{H^1(M)} dt' \big)
 \\&
 + \tilde K E\big(\int_0^t \max \{ 1, \|u_n(t',\cdot,\cdot)\|^2_{H^1(M)} \,dt' \} \big)
  +E\big( \int_0^t \|\Phi(\cdot,u_n)\|^2_{H^1(M)}  dt' \big),
\end{align*}
where $\tilde{k}, {\tilde K}$ are 
positive constants depending on $k, {K}$ from Lemma \ref{Lpar}. 
According to \eqref{H1}, we see that the right-hand side of the previous expression is bounded and thus
\begin{equation}
\label{H2}
 E\big( \| u_n(t,\cdot,\cdot)\|^2_{H^1(M)}\big)+ E\big(\int_0^t \| u_n(t,\cdot,\cdot)\|^2_{H^2(M)} dt' \big) \leq C \ \ \forall t\geq 0.
\end{equation}

\subsection*{Convergence}

{Using the splitting into $\mathrm{span}\{e_1,\dots,e_n\})$ and $\overline{\mathrm{span}\{e_k\}_{k>n}}$} and \eqref{H2} 
\begin{align}
\label{H1-R}
&E\left(\|(u_n-u_m)^{\bot}_n\|_{L^2([0,T]\times M)}\right)=E\left(\|(u_m)^{\bot}_n\|_{L^2([0,T]\times M)} \right)\\
&=  E\Big( \big[\int_0^T\sum\limits_{j=n+1}^m |\alpha^m_j(t,\cdot,\cdot)|^2 dt\big]^{1/2} \Big)
 = \frac{\lambda_N}{\lambda_N}E\Big( \big[\int_0^T\sum\limits_{j=n+1}^m |\alpha^m_j(t,\cdot,\cdot)|^2 dt\big]^{1/2} \Big)\nonumber \\ 
 & \leq \frac{1}{\lambda_N}E\Big( \big[\int_0^T\sum\limits_{j=n+1}^m \lambda_j^2|\alpha^m_j(t,\cdot,\cdot)|^2 dt\big]^{1/2} \Big) \leq \frac{1}{\lambda_N}E\left( \|u_m\|_{L^2([0,T]; H^1(M))}\right) \to 0 \nonumber 
 \end{align}
 as $N\to \infty$, and analogously
 \begin{equation}\label{H2-R}
 \begin{split}
 &E\left(\|(u_n-u_m)^{\bot}_n\|_{L^2([0,T];H^1(M))}\right)=E\left(\|(u_m)^{\bot}_n\|_{L^2([0,T];H^1(M))} \right)  \\ &\leq \frac{1}{\lambda_N}E\left( \|u_m\|_{L^2([0,T]; H^2(M))}\right) \to 0
\end{split}
\end{equation}
as $N\to \infty$.


Using the above and the linearity of the parabolic term, unlike the situation from the previous section, we can establish directly that  $(u_n)$ is a Cauchy sequence in $L_{{\bf P}}^2(\Omega; L^2([0,T];H^1(M)))$. Then, we shall prove continuity of the limit with respect to $t\in [0,T]$. To this end, we subtract equations \eqref{app-1-s} for $m\geq n \geq N\in \N$ to get for any $j\leq n\leq m$
\begin{equation}
\label{system-mn}
\begin{split}
d ({\alpha}_j^n-{\alpha}_j^m)&=\int_M (\mff(\mx,u_n)-\mff(\mx,u_m))\,\nabla e_j(\mx) \, dV(\mx) dt\\
& {-\int_M \Div\big(A_{\mx}\cdot (u_n-u_m)\big)(\nabla e_j)\, dV(\mx) dt}
\\&+\int_M \big(\Phi (\mx,u_n)-\Phi (\mx,u_m) \big) e_j(\mx) dV(\mx) dW_t.
\end{split}
\end{equation} From here, according to the It\^o lemma, we have
\begin{align*}
\frac{d |{\alpha}_j^n-{\alpha}_j^m|^2}{2}&= \int_M\big(\mff(\mx,u_n)- \mff(\mx,u_m) \big)\, \nabla \big(({\alpha}_j^n-{\alpha}_j^m)e_j(\mx) \big) dV(\mx) dt \\&
-\int_M \big(\Div(A_{\mx}\cdot ( u_n- u_m)\big)\nabla \big(({\alpha}_j^n-{\alpha}_j^m)e_j(\mx) \big) \, dV(\mx) dt\\
&+\big[\int_M \big(\Phi (\mx,u_n)-\Phi (\mx,u_m) \big) e_j(\mx) dV(\mx) \big]^2 dt\\
&+\int_M \big(\Phi (\mx,u_n)-\Phi (\mx,u_m) \big) ({\alpha}_j^n-{\alpha}_j^m) e_j(\mx) dV(\mx) dW_t.
\end{align*}
Keeping in mind \eqref{system-mn}, we get after summing the latter expressions for $j=1,\dots,n$:
\begin{equation}
\label{ISCL}
\begin{split}
&\frac{1}{2} d \int_M|u_n-u_m|^2 dV(\mx)=\int_M (\mff(\mx,u_n)-\mff(\mx,u_m)) \cdot\nabla (u_n-u_m) dV(\mx) dt \\
&\hspace*{1em}-\int_M \Div(A_{\mx}(u_n-u_m))\cdot\nabla (u_n-u_m) \, dV(\mx)dt  \\
&\hspace*{2em}+\int_M {(\Phi(\mx,u_n)-\Phi(\mx,u_m))^2} dV(\mx) dt \\
&\hspace*{3em}+ \int_M (\Phi(\mx,u_n)-\Phi(\mx,u_m)) (u_n-u_m) dV(\mx) dW_t\\
&\hspace*{4em}-\int_M (\mff(\mx,u_n)-\mff(\mx,u_m))\nabla \varphi^{\bot}_n \, dV(\mx) dt\\
&\hspace*{5em}+\int_M \Div(A_{\mx}(u_n-u_m))\cdot\nabla \varphi_n^\bot \, dV(\mx)dt   \\
&\hspace*{6em}-\int_M {\big[(\Phi(\mx,u_n)-\Phi(\mx,u_m))^\bot_n\big]^2} dV(\mx) dt\\ 
&\hspace*{7em}-\int_M (\Phi (\mx,u_n)-\Phi (\mx,u_m)) \varphi^{\bot}_n dV(\mx) dW_t,
\end{split}
\end{equation} where $\varphi^{\bot}_n=(u_n-u_m)^{\bot}_n=-(u_m)^{\bot}_n$. 

Let us first consider the terms containing $dt$ in \eqref{ISCL}. We have:
\begin{align*}
&\big| \int_M (\mff(\mx,u_n)-\mff(\mx,u_m)) \cdot\nabla (u_n-u_m) dV(\mx) dt \big| \\
&\leq K_1 \|\mff'\|^2_{\infty} \|u_n(t,\cdot,\omega)-u_m(t,\cdot,\omega)\|_{L^2(M)}^2 +\frac{c}{2} \|\nabla (u_n(t,\cdot,\omega)-u_m(t,\cdot,\omega))\|^2_{L^2(M)}. 
\end{align*}
Next, from Lemma \ref{Lpar} (i'), we have
\begin{align*}
&-\int_M \Div(A_{\mx}(u_n-u_m))(\nabla (u_n-u_m)) \, dV(\mx)dt\\&=\int_M \Div\Div(A_{\mx}(u_n-u_m)) (u_n-u_m) \, dV(\mx)dt \\& \leq 
-c \|\nabla (u_n(t,\cdot,\omega)-u_m(t,\cdot,\omega))\|^2_{L^2(M)}+C\|u_n(t,\cdot,\omega)-u_m(t,\cdot,\omega)\|^2_{L^2(M)}.
\end{align*} 
Furthermore,
\begin{equation*}
|\int_M {(\Phi(\mx,u_n)-\Phi(\mx,u_m))^2} dV(\mx)| \leq \|\Phi'\|_{\infty} \|u_n(t,\cdot,\omega)-u_m(t,\cdot,\omega)\|^2_{L^2(M)}.
\end{equation*} The rest of the terms containing $dt$ tend to zero as $N\to \infty$ (and thus $n,m\to \infty$) according to \eqref{H1-R} and \eqref{H2-R}. 

 From the above, we get after integrating \eqref{ISCL} over $[0,T]$ and applying the expectation operator (which makes the terms containing $dW_t$ vanish):
\begin{equation}
\label{ISCL-1}
\begin{split}
E\Big(\int_M &\frac{  |(u_n-u_m)(T,\mx,\omega)|^2}{2} dV(\mx)\Big) + \frac{c}{2}  E\Big(\int_0^T \int_M \|\nabla (u_n-u_m)\|_g^2 dV(\mx)\Big)\\& 
\leq E\big(\int_M \frac{  |(u_{0n}-u_{0m})(\mx,\omega)|^2}{2} dV(\mx) \big)\\ 
& \hspace*{10em}+\bar{C} E\big(\int_0^T \int_M | u_n-u_m |^2 dV(\mx) dt\big)+ c_n .
\end{split}
\end{equation}
where $c_n \to 0$ as $N\to \infty$.
%
%

From here, we see that the Gronwal lemma implies:
\begin{equation}
\label{H1-2S}
E\Big( \int_M |(u_n-u_m)(T,\mx,\omega)|^2 d\mx +\int_0^T\int_M |\nabla (u_n-u_m)(t,\mx)|^2 d\mx \Big) \leq C(T)c_n,
\end{equation}  for a locally finite function $C(T)$ depending on $T$ {and again $c_n\to 0$}. 

This proves that the sequence $(u_n)$ is Cauchy  in $L_{{\bf P}}^2(\Omega ;L^2([0,T];H^1(M)))$,  i.e.\ it converges towards some 
function $u\in L_{{\bf P}}^2(\Omega ;L^2([0,T];H^1(M)))$, {which is a solution to \eqref{eq:stoch_equation}}. 
Moreover, according to \eqref{H2}, we have $u\in L_{{\bf P}}^1(\Omega ;L^2([0,T];H^2(M)))$ as well.
It remains to prove that $u\in L^2_{{\bf P}}(\Omega; C^{1/2}((0,T); L^2(M)))$.


Consider the integral formulation of the stochastic differential equation \eqref{eq:stoch_equation}. We have  for every $\varphi\in H^2(M)$:
\begin{align*}
&\int_M \left( u(t+\Delta t,\mx,\omega)-u(t,\mx,\omega) \right) \varphi(\mx)dV(\mx )= 
\int_t^{t+\Delta t} \int_M \mff(\mx,u)\nabla \varphi(\mx) \, dV(\mx) dt\\
&+\int_t^{t+\Delta t} \int_M \Div(A_{\mx}\cdot u)(\nabla\vphi)\, dV(\mx) dt+\int_t^{t+\Delta t} \int_M  \Phi (\mx,u) \varphi(\mx) dV(\mx ) dW_t.
\end{align*} 
{To proceed from here, write $u=\sum_{j\in \N}\alpha_j e_j$, 
choose $\varphi(\mx)=e_k(\mx)$} and use integration by parts to get:
\begin{align*} 
& \alpha_k(t+\Delta t,\omega)-\alpha_k(t,\omega)= -
\int_t^{t+\Delta t} \int_M \Div \mff(\mx,u)  e_k(\mx) \, dV(\mx) dt'\\
&+\int_t^{t+\Delta t} \int_M \Div(\Div ( A_{\mx}(u) ) ) \,  e_k(\mx) dV(\mx) dt'+ \int_t^{t+\Delta t} \int_M \Phi (\mx,u) e_k(\mx) dV(\mx ) dW_{t'}.
\end{align*} 
We square the latter expression, find the expectation, and use the Jensen inequality to infer:
\begin{align*}
E\Big[\big( \alpha_k(t+\Delta t,&\omega)-\alpha_k(t,\omega) \big)^2 \Big] \\
&\leq {C}
\Big( \Delta t E\Big[ \int_t^{t+\Delta t}\left( \int_M \Div f(\mx,u) e_k(\mx) dV(\mx) \right)^2\, dt'\Big] \\
&\qquad\qquad+ \Delta t E\left[ \int_t^{t+\Delta t} \left(\int_M \Div(\Div ( A_{\mx}(u) ) ) e_k(\mx) dV(\mx ) \right)^2 dt'\right]\\
&\qquad\qquad\qquad
+E\Big[\Big(\int_t^{t+\Delta t}\int_M \Phi(\mx,u)e_k(\mx)\, dV(\mx) dW_{t'} \Big)^2\Big]\Big).
\end{align*} 
We divide the expression by $\Delta t$, use here the Ito isometry (Lemma \ref{IL2}), and sum the expression over $k\in \N$. We have:
\begin{equation}
\label{last1}
\begin{split}
&E\big[\|\frac{u(t+\Delta t,\cdot,\cdot)-u(t,\cdot,\cdot)}{\sqrt{\Delta t}}\|^2_{L^2(M)}\big] \\
& \hspace*{2em}
\leq C \Big( E\big[\int_t^{t+\Delta t}\|\Div f(\cdot, u(t',\cdot,\cdot))\|^2_{L^2(M)} dt' \big]\\
& \hspace*{2em}+ E\big[ \int_t^{t+\Delta t} \|\Div(\Div ( A_{\mx}\cdot u ) )\|^2_{L^2(M)}dt'\big]
\\
&\hspace*{2em}+ E \big[\frac{1}{\Delta t} \int_t^{t+\Delta t} \|\Phi(\cdot,u(t',\cdot,\cdot)) \|^2_{L^2(M)} dt'\big]\Big)\\
&\hspace*{2em} \leq C\Big( E\big[\|  u \|^2_{L^2((0,T); H^{2}(M)}\big]+ \sup\limits_{t\in (0,T)}E\big[ \|\Phi(\cdot,u(t,\cdot,\cdot))\|^2_{L^2(M)} \big] \Big) \leq \bar{C},
\end{split}
\end{equation}
since according to \eqref{H1} and \eqref{H2} 
\begin{align*}
&E\big[\int_t^{t+\Delta t}\|\Div f(\cdot, u(t',\cdot,\cdot))\|^2_{L^2(M)} dt' \big] \leq C E\big[ \|u\|^2_{L^2((0,T);H^{1}(M)}\big]\leq c <\infty  \\
&E\big[\int_t^{t+\Delta t} \|\Div(\Div ( A_{\mx}(u) ) )\|^2_{L^2(M)}\big] \leq C E\big[\|  u \|^2_{L^2((0,T);H^{2}(M)}\big] \leq c <\infty
\\
&  
E \big[\frac{1}{\Delta t} \int_t^{t+\Delta t} \|\Phi(\cdot,u(t',\cdot,\cdot)) \|^2_{L^2(M)} dt'\big] \leq \sup\limits_{t\in (0,T)}E \big[ \|\Phi(\cdot,u(t,\cdot,\cdot)) \|^2_{L^2(M)} \big] <\infty.
\end{align*}
\begin{equation}
\label{H-t}
E\big[\|u\|^2_{C^{1/2}([0,T);L^2(M)}\big] \leq \bar{C}<\infty \ \ \implies \ \ u\in L^2_{{\bf P}}(\Omega; C^{1/2}((0,T);L^2(M))).
\end{equation} This concludes the proof.

\subsection*{Acknowledgment} This research was supported by projects P30233 and P35508 of the Austrian Science Fund FWF.


\begin{thebibliography}{10}





\bibitem{AM} Abraham, R.,  Marsden, J.\ E.,
Foundations of mechanics. Second ed., Benjamin/Cummings, 1978.

\bibitem{BLf}  M.~Ben Artzi, P.~LeFloch, Well-posedness theory for geometry-compatible hyperbolic
conservation laws on manifolds, {\em Ann. I. H. Poincar\'e} {\bf 24} (2007), 989--1008.

\bibitem{benett} B.~Chow, P.~Lu, L.~Ni {\em Hamilton's Ricci Flow}, (Graduate Studies in Mathematics), American Mathematical Society 2006.

\bibitem{BoyFab} F.~Boyer, P.~Fabrie, 
Mathematical tools for the study of the incompressible Navier-Stokes equations and related models.
Applied Mathematical Sciences, 183. Springer, New York, 2013.

\bibitem{[CP]} J.~ Chazarain A.~ Piriou, {\em Introduction to the theory of linear PDEs}, 558 pages, 1982, North Holland.


\bibitem{CK} Q.~G.~Chen, K.~H.~Karlsen, Quasilinear Anisotropic Degenerate Parabolic Equations with Time-Space Dependent Diffusion Coefficients, {\em Comm.\  Pure and Applied Analysis} {\bf 4} (2005), 241--266.


\bibitem{Deck} T.~Deck, {\em Der It\^{o}-Kalkül. Einführung und Anwendungen}, 2006, Springer.






\bibitem{EK} M.~S.~Espedal and K.~H.~Karlsen, Numerical solution of reservoir flow models based on large time step operator splitting algorithms, {\em  Filtration in Porous Media and Industrial Applications}
(Cetraro, Italy, 1998), vol.\ {\bf 1734} of Lecture Notes in Mathematics, p.\ 9--77, Springer,
Berlin, 2000.

\bibitem{GK} L.~Galimberti, K.H.Karlsen, {\em Well-posedness theory for stochastically forced conservation laws on Riemannian manifolds
}, J. Hyp. Diff. Eq.  {\bf 16} (2019), 519--593. 






\bibitem{[G]} M.~P.~Gaffney, {\em Hilbert space methods in the theory of harmonic integrals}, Transactions of AMS {\bf 78} (1955), 426--444.

\bibitem{GKM} M.~Graf, M.~Kunzinger, D.~Mitrovic, {\em Well-posedness theory for degenerate parabolic equations on Riemannian manifolds}, J. Differential Equations {\bf 263} (2017), 4787--4825


\bibitem{GKOS} M.\ Grosser, M.\ Kunzinger, M.\ Oberguggenberger, R.\
Steinbauer, \emph{Geometric theory of generalized functions},
Kluwer, Dordrecht, 2001. 

\bibitem{Gess} B.~Gess, M.~Hofmanova, {Well-posedness and regularity for quasilinear degenerate parabolic-hyperbolic SPDE}, Ann. Probab. {\bf 46} (2018), 2495--2544.


\bibitem{[H]} E.~Hebey, {Sobolev spaces on Riemannian manifolds}, 120 pages, 1996, Springer-Verlag Berlin Heidelberg.


\bibitem{med} D.~ Holcman, Z~ Schuss, {\em Asymptotics of Elliptic and Parabolic PDEs: and their Applications in Statistical Physics, Computational Neuroscience, and Biophysics},  (Applied Mathematical Sciences) 1st ed. 2018, Springer.


\bibitem{HP} G.~ Huisken and A.~ Polden, {\em Geometric evolution equations for hypersurfaces, Calculus of variations and geometric evolution problems}, (Cetraro, 1996), Springer Verlag, Berlin, 1999, pp. 45--84.



\bibitem{kreut} M.\ Kreuter, {\em Sobolev Spaces of Vector-Valued Functions}, MSc.\ Thesis, University of Ulm, 2015,
{\tt https://www.uni-ulm.de/fileadmin/website\_uni\_ulm/mawi.inst.020\\
	/abschlussarbeiten/MA\_Marcel\_Kreuter.pdf}

\bibitem{MM} L.~ Martinazzi, C.~Mantegazza, {\em A note on quasilinear parabolic equations on manifolds}, Ann. Scuola Norm. Sup. Pisa. Cl. Sci. {\bf  XI}  (2012), 1--18.

\bibitem{LSU} O.~A.~Ladyzhenskaja, V.~A.~Solonnikov, N.~N.~Ural'ceva, (1968), {\em Linear and quasi-linear equations of parabolic type}, Translations of Mathematical Monographs 23, Providence, RI: American Mathematical Society.


\bibitem{Mar} J.\ E.\ Marsden,  Generalized Hamiltonian mechanics, Arch.\ Rat.\ Mech.\ Anal., 28(4) (1968) 323--361.


\bibitem{JMN}  J.~M.~Nordbotten,  M.~A.~Celia, {\em Geological Storage of CO$_2$:
Modeling Approaches for Large-Scale Simulation}, John Wiley and
Sons, 2011.

\bibitem{bio} B.~Perthame, {\em Parabolic Equations in Biology}, 199 pages, 2013, Springer, Cham.


\bibitem{polden} A.~Polden, {\em Curves and Surfaces of Least Total Curvature and Fourth-order Flows}, Ph.D. thesis, Mathematisches Institut, Univ. T\"ubingen, 1996, Arbeitsbereich Analysis Preprint Server Univ. T\"ubingen, http://poincare.mathematik.uni- tuebingen.de/mozilla/home.e.html.

\bibitem{S} J.~J.~Sharples, {\em Linear and quasilinear parabolic equations in Sobolev space}, J.Differential Equations {\bf 202} (2004), 111--142.

\bibitem{[S]} R.~T.~Seeley, {\em Complex powers of an elliptic operator}, Proc.Symp.Pure Math. {\bf 10} (1967), 288--307

\bibitem{tay} M.Taylor, {\em Partial Differential Equations III}, 711 pages, 2011, Springer.

\bibitem{[T]} M.~Taylor, {\em Pseudo-differential operators}, 160 pages, 1974, Springer-Verlag Berlin Heidelberg.

\bibitem{Oks} B.Oeksendal, {\em Stochastic differential equations}, 357 pages, 2003, Springer.

\bibitem{wald} R.\ Wald,
{\em General relativity}, University of Chicago Press, Chicago, IL, 1984.

\bibitem{link}  http://www.ltcc.ac.uk/media/london-taught-course-centre/documents/Applied-Computational-Methods---Notes-4.pdf



\end{thebibliography}
\end{document}